\documentclass[11pt]{article}
\usepackage[margin=1in]{geometry} 
\usepackage{amssymb,amsfonts,amsmath,bbm,mathrsfs,stmaryrd}
\usepackage{xcolor}
\usepackage{url}
\usepackage{cancel}

\usepackage{enumerate}

\usepackage[colorlinks,
             linkcolor=black!75!red,
             citecolor=blue,
             pdfproducer={pdfLaTeX},
             pdfpagemode=None,
             bookmarksopen=true
             bookmarksnumbered=true]{hyperref}

\usepackage{tikz}
\usetikzlibrary{arrows,calc,decorations.pathreplacing,decorations.markings,intersections,shapes.geometric,shapes,through,fit,shapes.symbols,positioning,decorations.pathmorphing}

\usepackage{braket}

\usepackage[amsmath,thmmarks,hyperref]{ntheorem}
\usepackage{cleveref}

\creflabelformat{enumi}{#2(#1)#3}

\crefname{section}{Section}{Sections}
\crefformat{section}{#2Section~#1#3} 
\Crefformat{section}{#2Section~#1#3} 

\crefname{subsection}{\S}{\S\S}
\AtBeginDocument{%
  \crefformat{subsection}{#2\S#1#3}%
  \Crefformat{subsection}{#2\S#1#3}%
}

\theoremstyle{plain}

\newtheorem{lemma}{Lemma}[section]
\newtheorem{proposition}[lemma]{Proposition}
\newtheorem{corollary}[lemma]{Corollary}
\newtheorem{theorem}[lemma]{Theorem}

\theoremstyle{nonumberplain}

\theoremstyle{plain}
\theorembodyfont{\upshape}
\theoremsymbol{\ensuremath{\blacklozenge}}

\newtheorem{definition}[lemma]{Definition}

\newtheorem{remark}[lemma]{Remark}

\newtheorem{notation}[lemma]{Notation}

\crefname{definition}{definition}{definitions}
\crefformat{definition}{#2definition~#1#3} 
\Crefformat{definition}{#2Definition~#1#3} 

\crefname{ex}{example}{examples}
\crefformat{example}{#2example~#1#3} 
\Crefformat{example}{#2Example~#1#3} 

\crefname{remark}{remark}{remarks}
\crefformat{remark}{#2remark~#1#3} 
\Crefformat{remark}{#2Remark~#1#3} 

\crefname{convention}{convention}{conventions}
\crefformat{convention}{#2convention~#1#3} 
\Crefformat{convention}{#2Convention~#1#3} 

\crefname{notation}{notation}{notations}
\crefformat{notation}{#2notation~#1#3} 
\Crefformat{notation}{#2Notation~#1#3}

\crefname{lemma}{lemma}{lemmas}
\crefformat{lemma}{#2lemma~#1#3} 
\Crefformat{lemma}{#2Lemma~#1#3} 

\crefname{proposition}{proposition}{propositions}
\crefformat{proposition}{#2proposition~#1#3} 
\Crefformat{proposition}{#2Proposition~#1#3} 

\crefname{corollary}{corollary}{corollaries}
\crefformat{corollary}{#2corollary~#1#3} 
\Crefformat{corollary}{#2Corollary~#1#3} 

\crefname{theorem}{theorem}{theorems}
\crefformat{theorem}{#2theorem~#1#3} 
\Crefformat{theorem}{#2Theorem~#1#3} 

\crefname{enumi}{}{}
\crefformat{enumi}{(#2#1#3)}
\Crefformat{enumi}{(#2#1#3)}

\crefname{assumption}{assumption}{Assumptions}
\crefformat{assumption}{#2assumption~#1#3} 
\Crefformat{assumption}{#2Assumption~#1#3} 

\crefname{equation}{}{}
\crefformat{equation}{(#2#1#3)} 
\Crefformat{equation}{(#2#1#3)}

\numberwithin{equation}{section}
\renewcommand{\theequation}{\thesection-\arabic{equation}}

\theoremstyle{nonumberplain}
\theoremsymbol{\ensuremath{\blacksquare}}

\newtheorem{proof}{Proof}
\newcommand\pf[1]{\newtheorem{#1}{Proof of \Cref{#1}}}

\newcommand\bC{{\mathbb C}}

\newcommand\bN{{\mathbb N}}

\newcommand\bS{{\mathbb S}}
\newcommand\bT{{\mathbb T}}

\newcommand\bZ{{\mathbb Z}}

\newcommand\cC{{\mathcal C}}
\newcommand\cD{{\mathcal D}}

\newcommand\cM{{\mathcal M}}

\newcommand\cP{{\mathcal P}}

\newcommand\cS{{\mathcal S}}

\newcommand\fG{{\mathfrak G}}

\newcommand\fK{{\mathfrak K}}

\newcommand\fg{{\mathfrak g}}

\newcommand\fo{{\mathfrak o}}
\newcommand\fp{{\mathfrak p}}

\newcommand\fs{{\mathfrak s}}

\newcommand\fgl{\mathfrak{gl}}
\newcommand\fsl{\mathfrak{sl}}

\newcommand\md{\mathrm{Mod}}

\DeclareMathOperator{\id}{id}
\DeclareMathOperator{\End}{\mathrm{End}}
\DeclareMathOperator{\Hom}{\mathrm{Hom}}

\newcommand\numberthis{\addtocounter{equation}{1}\tag{\theequation}}

\newcommand{\cat}[1]{\textsc{#1}}

\newcommand{\qedhere}{\mbox{}\hfill\ensuremath{\blacksquare}}

\title{Universal tensor categories generated by dual pairs}
\author{Alexandru Chirvasitu and Ivan Penkov}

\begin{document}

\date{}

\newcommand{\Addresses}{{
  \bigskip
  \footnotesize

  \textsc{Department of Mathematics, University at Buffalo, Buffalo,
    NY 14260-2900, USA}\par\nopagebreak
  \textit{E-mail address}: \texttt{achirvas@buffalo.edu}

  \medskip

  \textsc{Jacobs University Bremen, Campus Ring 1, 28759 Bremen, Germany}\par\nopagebreak
  \textit{E-mail address}: \texttt{i.penkov@jacobs-university.de}  

}}

\maketitle

\begin{abstract}
  Let $V_*\otimes V\rightarrow\mathbb{C}$ be a non-degenerate pairing of countable-dimensional complex vector spaces $V$ and $V_*$. The Mackey Lie algebra $\fg=\mathfrak{gl}^M(V,V_*)$ corresponding to this paring consists of all endomorphisms $\varphi$ of $V$ for which the space $V_*$ is stable under the dual endomorphism $\varphi^*: V^*\rightarrow V^*$. We study the tensor Grothendieck category $\mathbb{T}$ generated by the $\fg$-modules $V$, $V_*$ and their algebraic duals $V^*$ and $V^*_*$. This is an analogue of categories considered in prior literature, the main difference being that the trivial module $\mathbb{C}$ is no longer injective in $\mathbb{T}$. We describe the injective hull $I$ of $\mathbb{C}$ in $\mathbb{T}$, and show that the category $\mathbb{T}$ is Koszul. In addition, we prove that $I$ is endowed with a natural structure of commutative algebra. We then define another category $_I\mathbb{T}$ of objects in $\mathbb{T}$ which are free as $I$-modules. Our main result is that the category ${}_I\mathbb{T}$ is also Koszul, and moreover that ${}_I\mathbb{T}$ is universal among abelian $\mathbb{C}$-linear tensor categories generated by two objects $X$, $Y$ with fixed subobjects $X'\hookrightarrow X$, $Y'\hookrightarrow Y$ and a pairing $X\otimes Y\rightarrow \text{\textbf{1}}$ where \textbf{1} is the monoidal unit. We conclude the paper by discussing the orthogonal and symplectic analogues of the categories $\mathbb{T}$ and ${}_I\mathbb{T}$.
\end{abstract}


\vspace{.5cm}


\tableofcontents

\section{Introduction}

A tensor category for us is a symmetric, not necessarily rigid, $\mathbb{C}$-linear monoidal abelian category. In this paper we construct and study a tensor category which is universal as a tensor category generated by two objects $X$, $Y$ with fixed subobjects $X'\hookrightarrow X$, $Y'\hookrightarrow Y$ and endowed with a pairing $X\otimes Y\rightarrow {\bf 1}$, the object ${\bf 1}$ being the monoidal unit.

The simpler problem of constructing a universal tensor category generated just by two objects $X$, $Y$ endowed with pairing $X\otimes Y\rightarrow {\bf 1}$ was solved several years ago, and explicit constructions of such a category are given in \cite{SS} and \cite{DPS}. The construction in \cite{DPS} realizes this category as a category $\bT_{\fsl(\infty)}$ of representations of the Lie algebra $\fsl(\infty)$, choosing $X$ as the natural $\fsl(\infty)$-module $V$, and $Y$ as its restricted dual $V_*$. Motivated mostly by a desire to understand better the representation theory of the Lie algebra $\fsl(\infty)$, in \cite{PS1} a larger category was constructed, denoted $\widetilde{Tens}_{\fsl(\infty)}$, which contains also the algebraic dual modules $V^*$ and $V_*^*$. It is clear that the category $\widetilde{Tens}_{\fsl(\infty)}$ has a completely different flavor as its objects have uncountable length while $\bT_{\fsl(\infty)}$ is a finite-length category.

However, in \cite{PS2} the observation was made that the four representations $V$, $V_*$, $V^*$, $V_*^*$ generate a finite-length tensor category $\bT^{4}_{\mathfrak{gl}^M(V,V_*)}$ over the larger Lie algebra $\mathfrak{gl}^M(V,V_*)$, see Section \ref{se.prel}. We call this latter Lie algebra a Mackey Lie algebra as its introduction has been inspired by G. Mackey's work \cite{M}. The simple objects of $\bT^{4}_{\mathfrak{gl}^M(V,V_*)}$ were determined in \cite{Chi14}. Furthermore, in \cite{us0} the tensor category $\bT^{3}_{\mathfrak{gl}^M(V,V_*)}$, generated by $V$, $V_*$, and  $V^*$, was studied in detail. It was proved that $\bT^{3}_{\mathfrak{gl}^M(V,V_*)}$ is Koszul, and it was established that $\bT^{3}_{\mathfrak{gl}^M(V,V_*)}$ is universal as a tensor category generated by two objects $X$, $Y$ with a pairing $X\otimes Y\rightarrow {\bf 1}$, such that $X$ has a subobject $X'\hookrightarrow X$. Later, a vast generalization of the results of \cite{us0} was given in \cite{us1}: here a universal tensor category with two objects $X$, $Y$, a paring $X\otimes Y\rightarrow {\bf 1}$ and an arbitrary (possibly transfinite) fixed filtration of $X$ was realized as category of representations of a certain large Lie algebra.

A main difference of the category  $\bT^{4}_{\mathfrak{gl}^M(V,V_*)}$ with previously studied categories is that, as we show in the present paper, the injective hulls of simple objects are not objects of  $\bT^{4}_{\mathfrak{gl}^M(V,V_*)}$ but of the ind-completion $(\bT^{4}_{\mathfrak{gl}^M(V,V_*)})^{ind}$ which we denote simply by $\bT$. In particular, the trivial module has an injective hull $I$ in $\bT$ of infinite Loewy length, i.e. with an infinite socle filtration. Moreover, remarkably, $I$ has the structure of a commutative algebra. 

This leads us to the idea of considering the category ${}_I\bT$ of $I$-modules internal to $\bT$. The morphisms in this new category are morphisms of $\fgl^M(V,V_*)$-modules as well as of $I$-modules. The simple objects of ${}_I\bT$ are of the form $I\otimes M$ where $M$ is a simple module in $\bT$. 

A culminating result of the present paper is that the category ${}_I\bT$ has the universality property stated in the first paragraph of this introduction. The pairs $X'\hookrightarrow X$ and $Y'\hookrightarrow Y$ are realized respectively as $I \otimes V_*\subset I \otimes V^* $ and $I \otimes V\subset I \otimes V_*^*$, $I$ is the unity object, and the tensor product in ${}_I\bT$ is $\otimes_I$.

 Finally, in Section \ref{se:bcd} we study analogues of the tensor categories $\bT_{\fo(\infty)}$ and $\bT_{\fs\fp(\infty)}$ considered in \cite{DPS} and \cite{SS}. Consider a tensor category generated by a single object $X$ with a subobject $X'\hookrightarrow X$ and a pairing $X\otimes X\rightarrow \text{\textbf{1}}$. After identifying $V_*$ and $V$, our construction of the category ${}_I\bT$ yields a universal tensor category also in this setting. However, one can assume in addition that the pairing 
$X\otimes X\rightarrow \text{\textbf{1}}$ is symmetric or antisymmetric, which leads to new universality problems for tensor categories. 
 With this in mind, we introduce 
  $\bT^2_{\fo(V)}$ and $\bT^2_{\fs\fp(V)}$ where $\fo(V)$ and $\fs\fp(V)$ are respective orthogonal and symplectic Lie algebras of a countable-dimensional vector space $V$. In analogy with our previous constructions, we then produce appropriate categories $_{I'} \bT^2$ for $I'=I_{\fo(V)}$ and $I'=I_{\fs\fp(V)}$ and prove that these latter categories are universal in the respective new settings. Moreover, the categories ${}_{I_{\fo(V)}}\bT^2$ and ${}_{I_{\fs\fp(V)}}\bT^2$ are canonically equivalent as monoidal categories.

\subsection*{Acknowledgements}

A.C. was partially supported by NSF grants DMS-1801011 and DMS-2001128.

I.P. was partially supported by DFG grant PE 980-7/1.

\section{Preliminaries}\label{se.prel}

\subsection{Notation}
All vector spaces are defined over $\bC$ (more generally, we could work over an algebraically closed field of characteristic zero); similarly, all abelian categories and all functors between such are assumed $\bC$-linear, and we refer to \cite{groth} for general background on abelian/additive categories.


By $S^k X$ and $\Lambda^k X$ we denote respectively the $k$-th symmetric and exterior powers of a vector space $X$, and $S_n$ stands for the symmetric group on $n$ letters.

Once and for all we fix a non-degenerate pairing $V_*\otimes_{\bC} V\rightarrow \bC$ of countable-dimensional vector spaces $V$ and $V_*$. This pairing defines embeddings $V_*\subset V^*$, $V\subset V^*_*$, where $V^*=\Hom_{\bC}(V,\bC)$, $V_*^*=\Hom_{\bC}(V_*,\bC)$. For any vector space $M$ we set $M^*=\Hom_{\bC}(M,\bC)$. We abbreviate $\otimes_{\bC}$ as $\otimes$. By $\otimes$ we denote also tensor product in abstract tensor categories in the hope that this will cause no confusion.

Except in Section \ref{se:bcd}, $\fg$ will be the {\it Mackey Lie algebra} $\fgl^M$ {$(V,V_*)$} of \cite{PS2} associated to the pairing $V_*\otimes V\to \bC$. By definition, $$\fgl^M(V,V_*)=\{ \varphi\in \End V \mid \varphi^*(V_*)\subset V_*\},$$ where $\varphi^*:V^*\rightarrow V^*$ is the operator dual to $\varphi$. We will describe $\fg$ explicitly as Lie algebra of infinite matrices shortly. 

We set
\begin{equation*}
  W_*:=V^*/V_*,\quad W := V_*^*/V\quad\text{and}\quad F:=W_*\otimes W.  
\end{equation*}

There is an extension
\begin{equation}\label{eq:exatseqQ}
  0\to \bC\to Q\to F\to 0
\end{equation}
where $Q$ is defined as the quotient of $V^*\otimes V_*^*$ by the sum of the kernels of the pairings
\begin{equation*}
  V^*\otimes V\to \bC\quad \text{and}\quad V_*\otimes V_*^*\to \bC.
\end{equation*}
In Proposition \ref{pr.ann} below we prove that the extension \Cref{eq:exatseqQ} is non-splitting.


We model the actions of $\fg$ on the various modules mentioned above as follows:

\begin{itemize}
\item[$-$] $V_*^*$ consists of infinite column vectors with entries indexed by $\bN=\{0,1,\cdots\}$.
\item[$-$] $V\subset V^*_*$ consists of {\it finite} column vectors, i.e. those with at most finitely many non-zero entries.
\item[$-$] Dually, $V^*$ consists of $\bN$-indexed infinite row vectors.
\item[$-$] The elements of $V_*\subset V^*$ are precisely the finite row vectors. 
\item[$-$] $\fg$ consists of $\bN\times \bN$-matrices with finite rows and columns, acting on $V_*^*$ by left multiplication.
\item[$-$] Similarly, $\fg$ acts on $V^*$ as minus right multiplication. 
\item[$-$] $V\otimes V^*$$(=V^*\otimes V)$ consists of finite-rank $\bN\times \bN$ matrices, acted upon by $\fg$ by commutation.
\end{itemize}

We will frequently make use of Schur functors $\mathbb{S}_{\lambda}$ attached to Young diagrams $\lambda$. Often we write $X_{\lambda}$ instead of $\bS_{\lambda} X$ for a vector space $X$. Moreover, $S^k X=\mathbb{S}_{\rho} X$, $\Lambda^k X=\mathbb{S}_{\gamma}X$, where $\rho$, $\gamma$ are respectively a row and a column with $k$ boxes.

For Young diagrams $\lambda$, $\mu$, $\nu$ and $\pi$ we write
\begin{align*}
  L_{\lambda,\mu,\nu,\pi}&:=W_{*\lambda}\otimes V_{\mu,\nu}\otimes W_{\pi}\label{eq:s}\numberthis, \\
  J_{\lambda,\mu,\nu,\pi}&:=W_{*\lambda}\otimes V^*_{\mu}\otimes V_{\nu}\otimes W_{\pi},
\end{align*}
and similarly, for non-negative integers $l,m,n,p$ we set
\begin{align*}
  L_{l,m,n,p}&:=W_*^{\otimes l}\otimes V_{m,n}\otimes W^{\otimes p},\\
  J_{l,m,n,p}&:=W_*^{\otimes l}\otimes V^{*\otimes m}\otimes V_*^{*\otimes n}\otimes W^{\otimes p}\label{eq:j}\numberthis,
\end{align*}
where $V_{m,n}$ is the socle of $V^{*\otimes m}\otimes V^{\otimes n}$, i.e.
\begin{equation*}
  V_{m,n}=\bigoplus_{|\mu|=m,|\nu|=n}V_{\mu,\nu}^{m_{\mu,\nu}}
\end{equation*}
for appropriate multiplicities $m_{\mu,\nu}$. Here $|\lambda|$ denotes the degree (number of boxes) of a Young diagram $\lambda$. Finally, for any subscript $s$ of the form $(\bullet,\bullet,\bullet,\bullet)$ we set
\begin{equation}\label{eq:13}
  I_s:=I\otimes J_s,
\end{equation}
where $I$ is the object constructed below in \Cref{se.a}. 

\begin{definition}\label{def.thk}
  We refer to objects involving only the two outside diagrams $\lambda$ and $\pi$ as {\it thick} (or {\it purely thick} for emphasis) and those involving only the two middle diagrams as {\it thin}. Everything else is {\it mixed}. 
\end{definition}

It is essential to recall Corollary 4.3 in \cite{us0} which claims that $L_{\lambda,\mu,\nu,\pi}$ is a simple $\fg$-module, and implies that $L_{l,m,n,p}$ is a semisimple $\fg$-module.

The following remark will be used implicitly and repeatedly: given a short exact sequence
\begin{equation*}
  0\to x'\to x\to x''\to 0
\end{equation*}
in a tensor abelian category, the symmetric power $S^kx$ has a filtration
\begin{equation*}
  0=F_{-1}\subset F_0\subset \cdots\subset F_n = S^k x
\end{equation*}
with isomorphisms
\begin{equation*}
  F_j/F_{j-1} \cong S^{k-j} x'\otimes S^jx'' \text{ for }  0\le j\le k. 
\end{equation*}

\subsection{Plethysm}\label{subse:plet}


Given that $F=W\otimes W_*$ and we have to work with symmetric and exterior powers of $F$, we will have to understand how such powers decompose as direct sums of objects of the form $\mathbb{S}_{\lambda}W\otimes \mathbb{S}_{\mu}W_*$. The result applies to a tensor product $W\otimes W_*$ in {\it any} $\bC$-linear tensor category, so we work in this generality throughout the present subsection.

We call a partition $\lambda$ {\it special} if it satisfies the condition: all hooks of $\lambda$ with diagonal corner have horizontal and vertical arms of length $\mu_i-1$ and $\mu_i$ respectively, where $\mu_1> \mu_2 > \ldots > 0$ is a partition of $k$. We now recall the following result.

\begin{proposition}\label{pr.plth}
  Let $x$ and $y$ be two objects in a $\bC$-linear tensor category. We have the following decompositions:
  \begin{enumerate}[(a)]
  \item\label{item:1} $S^k(x\otimes y)$ is the direct sum of all objects of the form $\mathbb{S}_{\lambda}x\otimes \mathbb{S}_{\lambda}y$ as $\lambda$ ranges over all Young diagrams of degree $k$.
  \item\label{item:2} $\Lambda^k(x\otimes y)$ is the direct sum of all objects of the form $\mathbb{S}_{\lambda}x\otimes \mathbb{S}_{\lambda^{\perp}}y$ as $\lambda$ ranges over all Young diagrams of degree $k$, where $\lambda^{\perp}$ denotes the conjugate partition.
  \item\label{item:10} $S^k S^2 x$ is the direct sum of all $\bS_{\lambda}x$ for partitions $\lambda$ of degree $k$ with even parts , i.e. {\it even partitions}. 
  \item\label{item:11} $S^k \Lambda^2 x$ is $\bigoplus_{\substack{\text{even }\lambda \\ |\lambda|=k}}\bS_{\lambda^{\perp}}x$.
  \item\label{item:12} $\Lambda^k\Lambda^2 x$ is $\bigoplus_{\substack{\text{special }\lambda \\ |\lambda|=k}}\bS_{\lambda}x$.
  \item\label{item:13} $\Lambda^k S^2 x$ is $\bigoplus_{\substack{\text{special }\lambda \\ |\lambda|=k}}\bS_{\lambda^{\perp}}x$.    
  \end{enumerate}
\end{proposition}
\begin{proof}
  (\labelcref{item:1}) and (\labelcref{item:2}) are reformulations of the Cauchy identities in \cite[(6.2.8)]{ss-intro}. 
  The other four points paraphrase \cite[Example I.8.6]{macd-bk}.
\end{proof}

\subsection{Ordered Grothendieck categories}\label{sec:orderedgrothendieckcat}

We recall the following notion from \cite[Definition 2.3]{us1}.

\begin{definition}\label{def.ord-gr}
  Let $(\cP,\preceq)$ be a poset. An {\it ordered Grothendieck category} with underlying order $(\cP,\preceq)$ is a Grothendieck category $\cC$ together with objects $X_s$, $s\in \cP$ so that the following conditions hold.

\begin{enumerate}[(a)]
\item\label{item:3} The objects $X_s$ are {\it semi-artinian}, in the sense that all of their non-zero quotients have non-zero socles.
\item\label{item:5} Every object in $\cC$ is a subquotient of a direct sum of copies of various $X_s$.
\item\label{item:6} The simple subobjects in
  \begin{equation}\label{eq:14}
    \cS_s:=\{\text{isomorphism classes of simples in }\mathrm{soc}X_s\}
  \end{equation}
  are mutually non-isomorphic for distinct $s$ and they exhaust the simples in $\cC$.
\item\label{item:7} Simple subquotients of $X_s$ outside the socle $\mathrm{soc}X_s$ are in the socle of some $X_t$, $t\prec s$. 
\item\label{item:8} Each $X_s$ is a direct sum of objects with simple socle.
\item\label{item:9} Let $t\prec s$. The maximal subobject $X_{s\succ t}\subset X_{s}$ whose simple constituents belong to various $\cS_r$ for $s\succeq r\not\preceq t$ is the common kernel of a family of morphisms $X_s\to X_t$.
\end{enumerate}
\end{definition}

Ordered Grothendieck categories are well behaved in a number of ways. For instance (\cite[Corollary 2.6]{us1}):

\begin{proposition}
  The indecomposable injective objects in an ordered Grothendieck category $\cC$ are, up to isomorphism, precisely the indecomposable summands of the objects $X_s$ in \Cref{def.ord-gr}.  \qedhere
\end{proposition}

Recall (\cite[\S 3.2]{us0} or \cite[Definition 2.8]{us1}):

\begin{definition}\label{def:defect}
  For two elements $i\prec j$ in $\cP$ the {\it defect} $d(i,j)$ is the supremum of the set of non-negative integers $q$ for which we can find a chain
  \begin{equation*}
    i=i_0\prec \cdots\prec i_q=j.
  \end{equation*}
We put also $d(i,i):=0$. In the context of an ordered Grothendieck category as in \Cref{def.ord-gr} we adopt the simplified notation $d(S,T)$ for $d(s,t)$ when $S\in \cS_s$ and $T\in \cS_t$.  
\end{definition}

According to \cite[Proposition 2.9]{us0} ext functors in an ordered Grothendieck category exhibit the following ``upper triangular'' behavior.

\begin{proposition}\label{pr.def-ge}
  Let $S\in \cS_s$ and $T\in \cS_t$ be two simple objects in an ordered Grothendieck category. If $\mathrm{Ext}^q(S,T)\ne 0$ then $d(s,t)\ge q$. \qedhere
\end{proposition}

It is implicit in the statement that, in particular, we have $s\preceq t$ (see \cite[Lemma 3.8]{us0}). One of our goals will be to show that in the ordered Grothendieck category $\bT$ introduced in Section \ref{subse:ord-cat} below, we actually have equality
, and hence the category $\bT$ is Koszul in the following sense.

\begin{definition}\label{def.ksz}
  An ordered Grothendieck category is {\it Koszul} if for every $q\ge 0$ and every two simple objects $S\in \cS_s$ and $T\in \cS_t$ the canonical Yoneda composition map
  \begin{equation*}
    \bigoplus\mathrm{Ext}^1(S,U_1)\otimes \mathrm{Ext}^1(U_1,U_2)\otimes\cdots\otimes \mathrm{Ext}^1(U_{q-1},T)\to \mathrm{Ext}^q(S,T) 
  \end{equation*}
  is surjective, where the sum ranges over all isomorphism classes of simples $U_i$.  
\end{definition}

This mimics one of the characterizations of Koszul connected graded algebras, namely the requirement that the graded ext algebra $\mathrm{Ext}^*(k,k)$ of the ground field $k$ be generated in degree one (\cite[\S 2.1]{PP}). 

We introduce the following term to capture the desirable situation where defects precisely measure non-vanishing exts.

\begin{definition}\label{def.shrp}
  An ordered Grothendieck category is {\it sharp} if it satisfies the conclusion of \Cref{pr.def-ge} with equality rather than inequality. 
\end{definition}

The relevance of the concept to the preceding discussion follows from

\begin{theorem}\label{thm:kosz-gen}
  Assume that $\cC$ is an ordered Grothendieck category as in \Cref{def.ord-gr} 
  such that
  \begin{itemize}
  \item[$-$] the terms of the socle filtration of each indecomposable injective have finite length;
  \item[$-$] $\cC$ is sharp in the sense of \Cref{def.shrp}.    
  \end{itemize}
  Then $\cC$ is Koszul. 
\end{theorem}
\begin{proof}
  Fix arbitrary simple objects $S\in \cS_s$, $T\in \cS_t$ and a positive integer $q\ge 2$. It will be enough to show that the Yoneda composition
  \begin{equation*}
    \bigoplus_{\text{simple }U}\mathrm{Ext}^{q-1}(S,U)\otimes \mathrm{Ext}^1(U,T)\to \mathrm{Ext}^q(S,T)
  \end{equation*}
  is onto, since we can then proceed by induction on $q$. 

  Let
  \begin{equation*}
    0\to T\to I_T\to R_T\to 0
  \end{equation*}
  be the short exact sequence resulting from the embedding of $T$ into its injective hull $I_T$. This sequence constitutes an element of $\mathrm{Ext}^1(R_T,T)$, and Yoneda multiplication by that element induces an isomorphism
  \begin{equation*}
    \mathrm{Ext}^{q-1}(S,R_T)\cong \mathrm{Ext}^q(S,T). 
  \end{equation*}

  If $\mathrm{Ext}^q(S,T)=0$ there is nothing to prove. Otherwise, our sharpness assumption shows that $d(S,T)=q$. This together with the finite-length hypothesis then ensures that the socle of $Q_T$ is a finite direct sum of simples $U$ with $d(S,U)=q-1$. We furthermore have
  \begin{equation*}
     \mathrm{Ext}^{q-1}(S,\mathrm{soc}R_T) \cong \mathrm{Ext}^{q-1}(S,R_T)\text{,}
  \end{equation*}
  and hence every non-zero element of $\mathrm{Ext}^q(S,T)$ will be contained in the image of the Yoneda map
  \begin{equation*}
    \bigoplus_{U}\mathrm{Ext}^{q-1}(S,U)\otimes \mathrm{Ext}^1(U,T)\to \mathrm{Ext}^q(S,T)
  \end{equation*}
  where $U$ ranges over all isomorphism classes of simple constituents of $\mathrm{soc}R_T$. This finishes the proof.
\end{proof}

\subsection{Comodules}\label{subse:comodules}

\cite[Chapters I and II]{Sweedler} will provide sufficient background on coalgebras and comodules. For a coalgebra $C$ over a ground field $k$ we write $\cM^C$ for its category of right comodules and $\cM^C_{fin}$ for its category of finite-dimensional comodules. Since the Grothendieck categories we are interested in will turn out to be of the form $\cM^C$ for coalgebras $C$ we record in this short section a characterization of such categories from \cite{Tak77}.

The following is a paraphrase of \cite[Definition 4.1]{Tak77}, adapted in the context of {\it Grothendieck} (as opposed to plain abelian) categories.

\begin{definition}\label{def:loc-fin}
  A Grothendieck category is {\it locally finite} if it has a set of finite-length generators. 
\end{definition}

We then have the following recognition result for categories of comodules over fields (\cite[Theorem 5.1]{Tak77}):

\begin{theorem}\label{thm:tak}
  Let $k$ be a field. A $k$-linear Grothendieck category is equivalent to $\cM^C$ for a $k$-coalgebra $C$ if and only if it is locally finite in the sense of \Cref{def:loc-fin} and the endomorphism ring of every simple object is finite-dimensional over $k$.

  Moreover, in this case $\cM^C_{fin}$ can be identified with the subcategory of $\cC$ consisting of finite-length objects.
  \qedhere
\end{theorem}

The following notion (analogous to its dual- ring-theoretic version \cite[discussion preceding Theorem 2.1]{bass}) will also be relevant below.

\begin{definition}\label{def:perf}
  A coalgebra $C$ is {\it left semiperfect} if either of the following conditions, equivalent by \cite[Theorem 10]{lin}, holds{:}
  \begin{itemize}
  \item[$-$] every indecomposable injective right $C$-comodule is finite dimensional;
  \item[$-$] every finite-dimensional left $C$-comodule has a projective cover.
  \end{itemize}
\end{definition}

\subsection{Tensor categories}\label{subse:mon}

The categories we are most interested in are typically monoidal. The latter, in full generality, are covered for instance in \cite[Chapter XI]{cats}. In the context of abelian categories, we briefly recall the relevant definitions (see also \cite[\S 3.6]{us0}, where we make the same linguistic conventions). 

\begin{definition}\label{def:mon}
  A $\bC$-linear abelian category $\cC$ is {\it monoidal} if its monoidal structure has the property that $x\otimes \bullet$ and $\bullet\otimes x$ are exact endofunctors for every object $x$.

  If in addition the monoidal structure is symmetric, $(\cC,\otimes)$ is a {\it tensor category}.

  A {\it tensor functor} between tensor  categories is a $\bC$-linear symmetric monoidal functor.
\end{definition}

Note that this differs from conventions made elsewhere in the literature. In \cite[\S 1.2]{del-tan}, for instance, the term `cat\'egorie tensorielle' implies rigidity.

We occasionally write $(\cC,\otimes, {\bf 1})$ for a monoidal category to specify both the tensor product bifunctor and the monoidal unit object ${\bf 1}$.

\section{The categories $\bT$ and ${}_I\bT$}\label{se.a}

\subsection{Definition of the object $I$}

For every nonnegative integer $k$ we have a canonical embedding
\begin{equation}\label{eq:skqsubsetsk+1q}
  S^kQ\hookrightarrow S^{k+1}Q
\end{equation}
obtained as the composition

\begin{equation}\label{eq:sksk1}
  \begin{tikzpicture}[auto,baseline=(current  bounding  box.center)]
    \path[anchor=base] (0,0) node (1) {$S^kQ\cong S^kQ\otimes \bC$} +(4,0) node (2) {$S^kQ\otimes Q$} +(8,0) node (3) {$S^{k+1}Q$,};
    \draw[->] (1) to[bend left=0] node[pos=.5,auto] {$\scriptstyle \id\otimes\iota$} (2);
    \draw[->] (2) to[bend left=0] node[pos=.5,auto] {$\scriptstyle \text{multiplication}$} (3);
  \end{tikzpicture}
\end{equation}
where $\iota:\bC\to Q$ is the embedding defining $Q$ as an extension of $F$ by $\bC$. This gives rise to an exact sequence

\begin{equation}\label{eq:1}
  \begin{tikzpicture}[auto,baseline=(current  bounding  box.center)]
    \path[anchor=base] (0,0) node (01) {$0$} +(2,0) node (l) {$S^kQ$} +(4,0) node (m) {$S^{k+1}Q$} +(6,0) node (r) {$S^{k+1}F$} +(8,0) node (02) {$0$};
    \draw[->] (01) to[bend left=0] node[pos=.5,auto] {$\scriptstyle $} (l);
    \draw[->] (l) to[bend left=0] node[pos=.5,auto] {$\scriptstyle \iota$} (m);
    \draw[->] (m) to[bend left=0] node[pos=.5,auto] {$\scriptstyle \pi$} (r);
    \draw[->] (r) to[bend left=0] node[pos=.5,auto] {$\scriptstyle $} (02);
  \end{tikzpicture}
\end{equation}

Taking the colimit (or simply union)
\begin{equation}\label{eq:12}
  I:=\varinjlim_k S^kQ,
\end{equation}
we obtain a $\fg$-module that has an infinite ascending filtration representable schematically as 


\begin{equation}\label{eq:ifilt}
  \begin{tikzpicture}[auto,baseline=(current  bounding  box.center)]
    \path (0,0) node[name=1, minimum width=1cm, rectangle split, rectangle split parts=3, draw, rectangle split draw splits=true, align=center] 
    { $S^2F$
      \nodepart{two} $F$
      \nodepart{three} $\bC$
    } + (0,1.5) node {$\vdots$};
	\path (0.6,-0.9) node {$,$};
    
  \end{tikzpicture}
\end{equation}
where the boxes indicate the layers of the filtration.


The morphism
\begin{equation}\label{eq:2}
  \psi: I\to I/\bC \to F\otimes I
\end{equation}
to be defined below will play a central role in the sequel; we will occasionally write $\psi$ for the resulting factorization $I/\bC\to F\otimes I$ as well, leaving it to context to distinguish between the two possible meanings.

We obtain the morphism $\psi$ as a colimit $\varinjlim_k \psi^k$ where
\begin{equation}\label{eq:psik}
  \psi^k:S^{k}Q\to (S^{k}Q)/\bC \to F\otimes S^{k-1}Q.
\end{equation}
The latter map is defined as follows. First, recall that the symmetric algebra
\begin{equation*}
  S^{\bullet}Q = \bigoplus_{k\ge 0}S^kQ
\end{equation*}
has a graded Hopf algebra structure \cite[p.228]{Sweedler} making the degree-one elements {\it primitive}, i.e. such that the comultiplication
\begin{equation}\label{eq:delta}
  \Delta: S^{\bullet}Q\to S^{\bullet}Q\otimes S^{\bullet}Q
\end{equation}
is the unique algebra map defined by
\begin{equation*}
  S^{\bullet}Q\supset Q\ni v\mapsto v\otimes 1\oplus 1\otimes v\in (Q\otimes\bC)\oplus (\bC\otimes Q)\subset S^{\bullet}Q\otimes S^{\bullet}Q. 
\end{equation*}
The comultiplication \Cref{eq:delta} is a morphism of $\fg$-modules. By definition, \Cref{eq:psik} is given by

\begin{equation}\label{eq:3}
  \begin{tikzpicture}[auto,baseline=(current  bounding  box.center)]
    \path[anchor=base] (0,0) node (1) {$S^{k}Q$}
    +(4,1) node (2) {$Q\otimes S^{k-1}Q$}
    +(8,0) node (3) {$F\otimes S^{k-1}Q$}
    +(4,-1) node (4) {($S^{k}Q)/\bC$};
    \draw[->] (1) to[bend left=6] node[pos=.5,auto] {$\scriptstyle $} (2);
    \draw[->] (2) to[bend left=6] node[pos=.5,auto] {$\scriptstyle \pi\otimes\id$} (3);
    \draw[->] (1) to[bend right=6] node[pos=.5,auto,swap] {$\scriptstyle $} (4);
    \draw[->] (4) to[bend right=6] node[pos=.5,auto,swap] {$\scriptstyle $} (3);
    \draw[->] (1) to[bend right=0] node[pos=.5,auto] {$\scriptstyle \psi^k$} (3);    
  \end{tikzpicture}
\end{equation}
where
\begin{itemize}
\item[$-$] the upper left{-}hand arrow is the $Q\otimes S^{k-1}Q$-component of the comultiplication
  \begin{equation*}
    \Delta:S^kQ\to \bigoplus_{i=0}^k S^iQ\otimes S^{k-i}Q
  \end{equation*}
  described above.
\item[$-$] $\pi$ is the epimorphism fitting in \Cref{eq:1}.
\end{itemize}

To make sense of $\varinjlim_k \psi^k$ we would have to argue that these maps are compatible with the embeddings
\begin{equation*}
  \iota:S^kQ\hookrightarrow S^{k+1}Q
\end{equation*}
in \Cref{eq:1}, i.e. that all diagrams
\begin{equation*}
 \begin{tikzpicture}[auto,baseline=(current  bounding  box.center)]
  \path[anchor=base] 
  (0,0) node (l) {$S^kQ$} 
  +(6,0) node (r) {$F\otimes S^{k}Q$}
  +(3,.5) node (u) {$S^{k+1}Q$}
  +(3,-.5) node (d) {$F\otimes S^{k-1}Q$}
  ;
  \draw[->] (l) to[bend left=6] node[pos=.5,auto] {$\scriptstyle \iota$} (u);
  \draw[->] (u) to[bend left=6] node[pos=.5,auto] {$\scriptstyle \psi^{k+1}$} (r);
  \draw[->] (l) to[bend right=6] node[pos=.5,auto,swap] {$\scriptstyle \psi^k$} (d);
  \draw[->] (d) to[bend right=6] node[pos=.5,auto,swap] {$\scriptstyle \id\otimes\iota$} (r);  
 \end{tikzpicture}
\end{equation*}
commute (for arbitrary $k$). This can be seen by direct examination, fixing a basis $(v_{\alpha})_{\alpha}$ for $Q$ with a distinguished element $v_0=1\in \bC\subset Q$ and noting that the upper left{-}hand map in \Cref{eq:3} is defined on monomials by
\begin{equation}\label{eq:expldelta}
  v_{\alpha_1}\cdots v_{\alpha_k}\mapsto \sum_{i=1}^k v_{\alpha_i}\otimes v_{\alpha_1}\cdots v_{\alpha_{i-1}}v_{\alpha_{i+1}}\cdots v_{\alpha_k}. 
\end{equation}


\begin{lemma}\label{le.correct-ker}
  The kernel of $\psi:I\to F\otimes I$ is precisely $\bC$.
\end{lemma}
\begin{proof}
  The kernel of the upper right-hand map in \Cref{eq:3} is
  \begin{equation}\label{eq:ck-1}
   \mathbb{C}\otimes S^{k-1}Q\subset Q\otimes S^{k-1}Q,
  \end{equation}
  so we are in effect claiming that the preimage of \Cref{eq:ck-1} through the ``partial comultiplication''
  \begin{equation}\label{eq:partdelta}
    S^kQ\to Q\otimes S^{k-1}Q
  \end{equation}
  is $\bC\subset S^kQ$.

  This is easily seen from the explicit description \Cref{eq:expldelta} of the comultiplication \Cref{eq:partdelta}.
\end{proof}


\subsection{Order}\label{subse:ord}

Following (or rather amplifying) \cite{us0}, we order the quadruples $(l,m,n,p)$ of non-negative integers by setting
\begin{equation*}
  (l,m,n,p)\preceq (l',m',n',p')
\end{equation*}
precisely if
\begin{align*}
  l\ge l',\quad m\le m',&\quad p\ge p',\quad n\le n'\\
  l+m\le l'+m',\quad&\quad p+n\le p'+n'\numberthis \label{eq:ord}\\
  l+m-n-p&=l'+m'-n'-p'.
\end{align*}

For a quadruple $s=(l,m,n,p)$ we define a family $\Sigma_s$ of morphisms
\begin{equation*}
  I_s=J_s\otimes I\to J_{s'}\otimes I=I_{s'}
\end{equation*}
for various $s'\prec s$ as follows:
\begin{itemize}
\item[$-$] first, those of the form $\theta\otimes\id_I$ where $\theta\in \Theta_s$ as in \cite[\S 3.2]{us0}.
\item[$-$] secondly, $\id_{J_s}\otimes \psi_0$ where
  \begin{equation}\label{eq:16}
    \psi_0:I\to F\otimes I
  \end{equation}
  is the morphism in \Cref{eq:2}. 
\end{itemize}

The morphisms of the first type are such that their joint kernel is
\begin{equation*}
  L_{l,m,n,p}\otimes I\subseteq  J_{l,m,n,p}\otimes I=I_{l,m,n,p}. 
\end{equation*}
On the other hand the kernel of $\psi_0$ is $\bC\subset I$, and hence the joint kernel of $\Sigma_s$ is
\begin{equation}\label{eq:4}
  L_{l,m,n,p}\cong L_{l,m,n,p}\otimes \bC\subset J_{l,m,n,p}\otimes I=I_{l,m,n,p}. 
\end{equation}

We now want to argue that \Cref{eq:4} is precisely the inclusion of the socle:

\begin{proposition}\label{pr.soc}
  For every choice of non-negative integers $l$, $m$, $n$, and $p$ the object $L_{l,m,n,p}$ is the socle of $I_{l,m,n,p}$ via the inclusion \Cref{eq:4}. 
\end{proposition}

This will require some preparation. First, we have the following remark, in the spirit of \cite[Lemma 3.1]{us0}. 

\begin{lemma}\label{le.tens-dense}
  Let $\fG$ be a Lie algebra and $I\subseteq \fG$ be an ideal. Suppose $U\subseteq U'$ is an essential inclusion of $\fG/I$-modules and $D$ be a $\fG$-module on which $I$ acts densely. 
  Then the inclusion
  \begin{equation*}
    U\otimes D\subseteq U'\otimes D
  \end{equation*}
  is also essential.
\end{lemma}
\begin{proof}
  Let $u\otimes w\in U'\otimes D$ be a non-zero simple tensor. We claim that for every $\overline{x}\in \fG/I$, there is an $x\in \fG$ such that:
  \begin{itemize}
  \item[$-$]$x$ has image $\overline{x}$ modulo $I$;
  \item[$-$] $x(u\otimes w)=\overline{x}u\otimes w$.   
  \end{itemize}
  To see this, first choose an arbitrary $y\in \fG$ with image $\overline{x}$ in $\fG/${$I$}. We have
  \begin{equation*}
    y(u\otimes w) = \overline{x}u\otimes w + u\otimes yw.
  \end{equation*}
  On the other hand, by the density assumption there is some $a\in I$ such that $aw=yw$, and we can simply set $x=y-a$.

  Now let $w_1$ up to $w_k$ be linearly independent vectors in $D$ and consider an  element
  \begin{equation*}
    f:=\sum_{i=1}^k u_i\otimes w_i\in U'\otimes D.
  \end{equation*}
Using the fact that $U\subseteq U'$ is essential and the claim above, there is an element $x\in \fG$ such that 
\begin{equation}\label{eq:5}
  xf = \sum_{i=1}^k q_i\otimes w_i
\end{equation}
with $q_1\in U$. We can now repeat the procedure, acting on the right-hand side of \Cref{eq:5} with another element $x'\in \fG$ so that the resulting sum is of the same form, with $q_1$ and $q_2$ in $U$. By recursion, we will eventually obtain an element of the form
\begin{equation*}
  \sum_i q_i\otimes w_i\in U\otimes D
\end{equation*}
in the $U(\fG)$-submodule of $U'\otimes D$ generated by $f$. This finishes the proof.
\end{proof}

The following result will require some additional conventions and elaboration. Recall that $Q$ is the quotient
  \begin{equation*}
    V^*\otimes V^*_*/(\text{traceless tensors in }V^*\otimes V + V_* \otimes   V^*_*).
  \end{equation*}
  We noted above that we identify the original space $ V^*\otimes V^*_*$ with finite-rank infinite matrices, and hence the quotient consists of equivalence classes of such matrices, where two are declared equivalent whenever they differ along at most finitely many rows or columns (and this is extended minimally to an equivalence relation).

We fix a basis $\{e_{\alpha}\}_{\alpha\in A}$ of $Q$ as follows: 

\begin{itemize}
\item[$-$] $e_0=1\in \bC$.
\item[$-$] All other basis elements are classes of rank-1 matrices of the form $v^* \otimes  v $ for $v^*\in V^*$ and   $v\in V^*_*$.   
\end{itemize}

\begin{lemma}\label{le.prscr-ann}
  Let $X\subset V_*^*-V$ and $X^*\in V^*-V_*$ be finite subsets, linearly independent modulo $V$ and respectively $V_*$. Fix $x_0\in X$ and $x_0^*\in X^*$. There is an element $g\in \fg$ such that
  \begin{itemize}
  \item[$-$] $gx\in V$ for all $x\in X$,   
  \item[$-$] $gx^*\in V_*$ for all $x^*\in X^*$,
  \item[$-$] $g(  x_0^* \otimes  x_0)$ has non-zero trace,
  \item[$-$] $g(x^*\otimes x )$ has zero trace for all other choices of $x\in X$ and $x^*\in X^*$.
  \end{itemize}
\end{lemma}
\begin{proof}
  The conclusion will follow from the remark that $\fg$ acts densely on sets $X\cup X^*$, i.e. that given $x\in X$ and $x^*\in X^*$, the vectors $gx\in V \text{ and }gx^* \in V_*$ can be prescribed arbitrarily. Keeping this in mind, we can then find $g\in \fg$ such that
  \begin{itemize}
  \item[$-$] $gx=0$ for all $x\in X$, 
  \item[$-$] $gx^*=0$ for all $x\in X^*\setminus\{x_0^*\}$, 
  \item[$-$] the inner product of $gx_0^*$ with every $x\in X\setminus\{x_0\}$ vanishes, 
  \item[$-$] the inner product of $gx_0^*$ with $x_0$ does not vanish.
  \end{itemize}
  This choice will meet the requirements of the statement, hence the conclusion.
\end{proof}

\begin{proposition}\label{pr.ann}
  Let $F\subset W\cup W_*$ be a finite set of vectors and
  \begin{equation*}
   \fK_{F}:=\mathrm{Ann}_{\fg}F\subset\fg 
  \end{equation*}
  be the Lie subalgebra that annihilates all elements of $F$. Then for every positive integer $k$ the inclusion
  \begin{equation*}
    \bC\subset S^kQ
  \end{equation*}
  is essential over $\fK=\fK_F$.
\end{proposition}
\begin{proof}
  We have to show that the $\fK$-submodule generated by any non-zero element of $S^kQ$ intersects $\bC$. We fix a basis $\{e_{\alpha}\}_{\alpha\in A}$ of $S^kQ$ containing $e_0=1\in \bC$, as in the discussion preceding the statement of the present result. If we put a total order $\le $ on the index set $A$, the elements
  \begin{equation}\label{eq:10}
    e_{\bf t}=e_{\alpha_1}\cdots e_{\alpha_k}:=\sigma(e_{\alpha_1}\otimes \cdots\otimes e_{\alpha_k})
  \end{equation}
  for tuples
  \begin{equation*}
    {\bf t}=(\alpha_1,\cdots,\alpha_k),\ \alpha_1\le \cdots\le \alpha_k\in A
  \end{equation*}
  form a basis of $S^kQ$, where $\sigma$ denotes the symmetrization operator on $Q^{\otimes k}$.

  We assign $1=e_0$ degree zero and every other $e_{\alpha}$ degree $1$, thus allowing us to define a degree between $0$ and $k$ for each element \Cref{eq:10} and by extension for each $x\in S^kQ$, as the largest degree of a basis element \Cref{eq:10} appearing in a decomposition of $x$.

  We can now prove the claim that
  \begin{equation*}
    \bC\subseteq U(\fK)x
  \end{equation*}
  by induction on the degree of $x$. Since the base case $\deg(x)=0$ requires no proof, we focus on the induction step.

  Decompose
  \begin{equation}\label{eq:11}
    x=\sum c_{t}e_{t},\ c_{t}\ne 0,
  \end{equation}
  with $\deg(x)>0$. By \Cref{le.prscr-ann} we can arrange for an element $g\in \fK$ such that
  \begin{itemize}
  \item[$-$] $g$ annihilates all elements of $F\subset W\cup W_*$,
  \item[$-$] $g$ sends one of the elements $1\ne e_{\alpha}$ appearing among the tensorands in \Cref{eq:11} to a non-zero scalar multiple of $e_0$,
  \item[$-$] $g$ annihilates all other $e_{\alpha}$ appearing in \Cref{eq:11}.
  \end{itemize}

  Clearly then
  \begin{equation*}
    \deg(gx) = \deg(x)-1,
  \end{equation*}
  and we can conclude the argument by using the induction hypothesis.
\end{proof}

\pf{pr.soc}
\begin{pr.soc}
  We know that $L_{l,m,n,p}$ is semisimple by \cite[Corollary 4.3]{us0}, so it suffices to show that \Cref{eq:4} is essential. 

  Since $W$, $W_*$ and $I$ are trivial as $\fsl(\infty)$-modules and
  \begin{equation*}
    V_{m,n}\subset V^{*\otimes m}\otimes V^{\otimes n}
  \end{equation*}
  is the socle over $\fsl(\infty)$, it follows by restricting to the latter subalgebra of $\fg$ that the inclusion
  \begin{equation*}
    L_{l,m,n,p}\otimes I\subset J_{l,m,n,p}\otimes I
  \end{equation*}
  is essential, reducing the goal to proving that so is the inclusion
  \begin{equation}\label{eq:6}
    L_{l,m,n,p}\subset L_{l,m,n,p}\otimes I. 
  \end{equation}
  We can simplify this further: recall that
  \begin{equation*}
    L_{l,m,n,p}=W_*^{\otimes l}\otimes V_{m,n}\otimes W^{\otimes p}.
  \end{equation*}
  Now apply \Cref{le.tens-dense} in the following setup:
  \begin{itemize}
  \item[$-$] $\fG=\fg$ and $I=\fsl(\infty)$;
  \item[$-$] the inclusion $U\subseteq U'$ is
    \begin{equation}\label{eq:7}
      W_*^{\otimes l}\otimes W^{\otimes p}\subset W_*^{\otimes l}\otimes W^{\otimes p}\otimes I;
    \end{equation}
  \item[$-$] $D$ is any of the simple direct summands $V_{\mu,\nu}$ of $V_{m,n}$.    
  \end{itemize}
  \Cref{le.tens-dense} then shows that \Cref{eq:6} is indeed 
  essential provided \Cref{eq:7} is. In other words, it is enough to consider $m=n=0$ in \Cref{eq:6}. Since $I$ is the union of $S^kQ$ as $k\to \infty$, it will furthermore be sufficient to argue that, for every $l$, $p$ and $k$, the inclusion
  \begin{equation}\label{eq:9}
    W_*^{\otimes l}\otimes W^{\otimes p}\subset W_*^{\otimes l}\otimes W^{\otimes p}\otimes S^kQ
  \end{equation}
  is essential.

  We can now conclude via \Cref{pr.ann}: an arbitrary non-zero element of the right-hand side of \Cref{eq:9} is of the form
  \begin{equation*}
    \sum_{i=1}^r e_i\otimes v_i
  \end{equation*}
  where $e_i=\sum_{j} a_{i,j} \otimes b_{i,j}$ are linearly independent elements of $W_*^{\otimes l}\otimes W^{\otimes p}$ and $v_i\in S^kQ$. Now let $F\subset W\cup W_*$ be the finite set of vectors $\{a_{i,j}, b_{i,j}\}$ and consider the annihilator $\fK_F$ of $F$, as in \Cref{pr.ann}.

  The Lie algebra $\fK_F$ leaves the subspace
  \begin{equation*}
    \left(\bigoplus_{i=1}^r \bC e_i \right)\otimes S^kQ
  \end{equation*}
  invariant and its action makes that space isomorphic to $(S^kQ)^{\oplus r}$. The conclusion thus follows from \Cref{pr.ann}.
\end{pr.soc}

\subsection{Simple objects and their endomorphism algebras}\label{subse:endo}

The main result of the present subsection is the following (presumably expected) claim.

\begin{theorem}\label{th:scalaralgs}
The simple objects $L_{\lambda,\mu,\nu,\pi}$ are mutually non-isomorphic and have scalar endomorphism algebras.   
\end{theorem}

The arguments, which require some groundwork, will be in the spirit of those used in the proof of the analogous statement \cite[Theorem 3.5]{us0}. First, recall \cite[Lemma 3.1]{us0}:

\begin{lemma}\label{le:dense-us0}
  Let $\fG$ be a Lie algebra and $J\subseteq \fG$ be an ideal. Suppose $U$, $U'$ are two $\fG/J$-modules and $D$ a $\fG$-module on which $J$ acts densely and irreducibly with $\End_J D=\bC$. Then, the inclusion
  \begin{equation*}
    \Hom_{\fG}(U,U')\ni f\mapsto f\otimes \id\in \Hom_{\fG}(U\otimes D,U'\otimes D)
  \end{equation*}
  is an isomorphism.
  \qedhere
\end{lemma}

\begin{proposition}\label{pr:gpendoalgs}
  For any two non-negative integers $l$, $p$ and Young diagrams $\mu$, $\nu$ the endomorphism algebra in $\bT$ of the object $W_{*l}\otimes V_{\mu,\nu}\otimes W_{p}$ is the group algebra $\bC [S_l\times S_p]$, with the two symmetric-group factors acting on the two outer tensorands. 
\end{proposition}
\begin{proof}
  We apply \Cref{le:dense-us0} to the ideal
  \begin{equation*}
    \fsl(\infty)=:J\subset \fG:=\fg,
  \end{equation*}
  with
  \begin{equation*}
    U=U'=W_{*l}\otimes W_p\quad\text{and}\quad D=V_{\mu,\nu}. 
  \end{equation*}
  The density of the action of $\fsl(\infty)$ on $V_{\mu,\nu}$ and the isomorphism $\End_{\fsl(\infty)} V_{\mu,\nu}\cong \bC$ follow by realizing the object $V_{\mu,\nu}$ as a colimit of irreducible $\fsl_n$-modules, while establishing an isomorphism
  \begin{equation*}
    \End_{\fg} ( W_{*l}\otimes W_p)\cong \bC[S_l\times S_p]
  \end{equation*}
  is entirely parallel to \cite[Proposition 3.2]{us0}, whose proof we do not reprise here. \Cref{le:dense-us0} then implies the desired isomorphism
  \begin{equation*}
    \End_{\fg} (W_{*l}\otimes V_{\mu,\nu}\otimes W_p)\cong  \End_{\fg} (W_{*l}\otimes W_p)\cong \bC[S_l\times S_p].
  \end{equation*}
\end{proof}

\pf{th:scalaralgs}
\begin{th:scalaralgs}
  According to \Cref{pr:gpendoalgs}, the endomorphism algebra
  \begin{equation}\label{eq:endss}
    \End_{\fg} (W_{*l}\otimes V_{\mu,\nu}\otimes W_{p})\cong \bC[S_l\times S_p]\cong \bC [S_l ]\otimes \bC [S_p]
  \end{equation}
  is semisimple. Since the tensor products
    $c_{\lambda}\otimes c_{\pi}\in \bC [S_l]\otimes \bC [S_p ]$
  of Young projectors ranging over diagrams with $|\lambda|=l,~|\pi|=p$,  
  form a complete system of equivalence classes of minimal idempotents in \Cref{eq:endss} under inner conjugation, the semisimple object $W_{*l}\otimes V_{\mu,\nu}\otimes W_{p}$ has simple constituents isomorphic to
  \begin{equation*}
    (c_\lambda\otimes c_{\pi})\left(W_{*l}\otimes V_{\mu,\nu}\otimes W_{p}\right) \cong L_{\lambda,\mu,\nu,\pi},
  \end{equation*}
  with  $L_{\lambda,\mu,\nu,\pi}$ not isomorphic to $ L_{\lambda',\mu,\nu,\pi'}$ for distinct pairs $(\lambda,\pi)\ne (\lambda',\pi')$ because
  \begin{equation*}
    c_{\lambda}\otimes c_{\pi}\text{ is not inner-conjugate to }c_{\lambda'}\otimes c_{\pi'}.
  \end{equation*}
  
  Furthermore, using Proposition \ref{pr:gpendoalgs}, we calculate 
  \begin{align*}
  \End_{\fg}L_{\lambda,\mu,\nu,\pi}&\cong \End_{\fg}\left((c_\lambda\otimes c_{\pi})\left(W_{*l}\otimes V_{\mu,\nu}\otimes W_{p}\right) \right)\\
  &\cong (c_\lambda\otimes c_{\pi})\End_{\fg}\left(W_{*l}\otimes V_{\mu,\nu}\otimes W_{p}\right)(c_\lambda\otimes c_{\pi})\\
  &\cong (c_\lambda\otimes c_{\pi})\bC[S_l\times S_p](c_\lambda\otimes c_{\pi}).
  \end{align*}
  Since $(c_{\lambda}\otimes c_{\pi})\bC[S_l\times S_p]=X$ is a simple $\bC[S_l\times S_p] $-module, we have $(c_{\lambda}\otimes c_{\pi})\bC[S_l\times S_p](c_{\lambda}\otimes c_{\pi})=\End_{\bC[S_l\times S_p]} X=\bC$, and the statement is proved.

\end{th:scalaralgs}

\subsection{The category $\bT$}\label{subse:ord-cat}

\begin{definition}
  The category $\bT$ is the smallest full 
  tensor Grothendieck subcategory of the category ${}_{\fg}\md$ of $\fg$-modules, closed under taking subquotients, and containing
  \begin{itemize}
  \item[$-$] the objects $J_s$ of \Cref{eq:j} for quadruples $s=(l,m,n,p)$;
    \item[$-$] the object $I$ of \Cref{eq:12}. 
  \end{itemize}
\end{definition}

The indices $s=(l,m,n,p)$ form a poset $(\cP,\preceq)$ under the ordering introduced in \Cref{subse:ord}. Keeping that in mind, we have

\begin{proposition}\label{pr.is-ordered}
$\bT$ is an ordered Grothendieck category in the sense of \Cref{def.ord-gr}.   
\end{proposition}
\begin{proof}
  We have to check the conditions listed in \Cref{def.ord-gr}. Here, the objects $X_s$ will be the objects $I_s$ from \Cref{eq:13} for $s=(l,m,n,p)$.

  {Condition \emph{(\labelcref{item:3})}.} This follows from the fact that all $I_s$ have countable filtrations whose subquotients are simple objects of the form $L_{\lambda,\mu,\nu,\pi}$ as in \Cref{eq:s}. The latter is clear as the objects $J_s$ have finite length and $I$ has the filtration \Cref{eq:12}.


  {Condition \emph{(\labelcref{item:5})}.} This holds essentially by construction. 

  {Condition \emph{(\labelcref{item:6})}} is a consequence of \cite[Proposition 5.4]{us0}.  

  {Condition \emph{(\labelcref{item:7})}.} Once more, we filter $I_s=J_s\otimes I$ by first refining the socle filtration maximally of $J_s$ and then tensor by some maximal refinement of the filtration \Cref{eq:12}. 

  The successive subquotients
  \begin{equation*}
    S^{k+1}Q/S^kQ \cong S^{k+1}F= S^{k+1}(W\otimes W_*)
  \end{equation*}
  of \Cref{eq:12} can be decomposed as sums of objects of the form $W_{\lambda}\otimes W_{*\lambda}$ by part \emph{(\labelcref{item:1})} of \Cref{pr.plth}. Hence, tensoring a simple subquotient $S\in \cS_s$
  of $J_s$ for some $s=(l,m,n,p)$ by such an object has the effect of increasing $l$ and $m$ by the same amount, thus resulting in some $t\prec s$ according to our ordering \Cref{eq:ord}.

  It thus remains to argue for simple subquotients of
  \begin{equation*}
    J_s=J_{l,m,n,p}=W_*^{\otimes l}\otimes V^{*\otimes m}\otimes V_*^{*\otimes n}\otimes W^{\otimes p}
  \end{equation*}
  instead. In this case though the filtration of $J_s$ is obtained either by surjecting one of the tensorands $V^*$ onto $W_*=V^*/V_*$ or similarly, one of the tensorands $V_*^*$ onto $W$, or by evaluating some $V_*$ against some $V$.

  All of these procedures map into $J_t$ for $t\prec s$, hence the conclusion.

  {Condition \emph{(\labelcref{item:8})}.} Indeed, for $s=(l,m,n,p)$ the object $I_s$ decomposes as
  \begin{equation*}
    I_s = \bigoplus_{\lambda,\mu,\nu,\pi}I_{\lambda,\mu,\nu,\pi}
  \end{equation*}
  where the sum ranges over $|\lambda|=l$, $|\mu|=m$, etc. The summands have simple respective socles $L_{\lambda,\mu,\nu,\pi}$ by \Cref{pr.soc}. 

  {Condition \emph{(\labelcref{item:9})}.} The morphisms $I_s\to I_t$ will be compositions of the obvious ones:
  \begin{itemize}
  \item[$-$] projecting one of the tensorands $V^*$ of $I_s=J_s\otimes I$ onto $W_*$;
  \item[$-$] the dual analogue, $V^*_*\to W$;
  \item[$-$] the surjection defining $Q$,
    \begin{equation}\label{eq:partpair}
      V^*\otimes V^*_*\to Q\subset I;
    \end{equation}
  \item[$-$] applying the morphism $I\to F\otimes I$ in \Cref{eq:16} to the tensorand $I$ of $I_s$. 
  \end{itemize}
  The verification that the joint kernel of these maps is as claimed routine.  
\end{proof}

In particular, \cite[Proposition 2.5]{us1} and \Cref{pr.soc} together prove

\begin{theorem}\label{thm:hulls}
  For every quadruple $(\lambda,\mu,\nu,\pi)$ of Young diagrams, $I_{\lambda,\mu,\nu,\pi}$ is an injective hull of $L_{\lambda,\mu,\nu,\pi}$.
\end{theorem}

We record the following observation.

\begin{lemma}\label{le.defct}
  Let $i=(l,m,n,p)$ and $i'=(l',m',n',p')$ be two elements of the poset described in \Cref{eq:ord}. Then $i\preceq i'$ implies
  \begin{equation*}
    d(i,i') = l-l' + n'-n.
  \end{equation*}
\end{lemma}

\begin{remark}\label{re.symmetry}
  The category $\bT$ is symmetric with respect to the simultaneous interchange $V \leftrightarrow V_*$, $V^*\leftrightarrow V_*^*$. Numerically, this corresponds to $l\leftrightarrow p$ and $n\leftrightarrow m$. \Cref{le.defct} is compatible with this transformation: according to the last condition in \Cref{eq:ord} we have
  \begin{equation*}
    l-l' + n'-n = p-p' + m'-m,
  \end{equation*}
  so we could have substituted $p-p' + m'-m$  
  for 
  $l-l' + n'-n$ in \Cref{le.defct}. 
\end{remark}

\subsection{Injective resolutions}

We will now show that $\bC$ admits an injective resolution in $\bT$
\begin{equation}\label{eq:15}
  0\to \bC\to I_0\to I_1\to \cdots
\end{equation}
with 
\begin{equation*}
  I_j\cong \Lambda^j F\otimes I. 
\end{equation*}
We will also see that $I_j/\mathrm{Im}(I_{j-1})$ admits an ascending filtration with layers
\begin{equation*}
  \begin{tikzpicture}
    \path (0,0) node[name=1, minimum width=1cm, rectangle split, rectangle split parts=3, draw, rectangle split draw splits=true, align=center] 
    { $\mathbb{S}_{(3,1,\cdots,1)}F$
      \nodepart{two} $\mathbb{S}_{(2,1,\cdots,1)}F$
      \nodepart{three} $\mathbb{S}_{(1,1,\cdots,1)}F$
    } + (0,1.5) node {$\vdots$};
    
  \end{tikzpicture}  
\end{equation*}
where each diagram has $j+1$ rows.

To streamline the notation for such Young diagrams we denote by $(l,j\times 1)$ the diagram with a row of length $l$ and $j$ single-box rows.


To define the maps
\begin{equation}\label{eq:psij}
  \psi_j:\Lambda^j F\otimes I\to \Lambda^{j+1} F\otimes I
\end{equation}
we mimic the procedure used in the definition of \Cref{eq:16}. In fact, that notation will be compatible with \Cref{eq:psij}, in that we will recover that earlier map by setting $j=0$ in the latter. As before, \Cref{eq:psij} will be a colimit as $k\to \infty$ of maps
\begin{equation}\label{eq:psijk}
  \psi^k_j:\Lambda^j F\otimes S^kQ\to \Lambda^{j+1} F\otimes S^{k-1}Q.
\end{equation}

The analogue of \Cref{eq:3} in this context is
\begin{equation}\label{eq:17}
  \begin{tikzpicture}[auto,baseline=(current  bounding  box.center)]
    \path[anchor=base] (0,0) node (1) {$\Lambda^j F\otimes S^{k}Q$}
    +(4,1) node (2) {$\Lambda^jF\otimes Q\otimes S^{k-1}Q$}
    +(8,1) node (3) {$\Lambda^jF\otimes F\otimes S^{k-1}Q$}
    +(12,0) node (4) {$\Lambda^{j+1}F\otimes S^{k-1}Q$};
    \draw[->] (1) to[bend left=6] node[pos=.5,auto] {$\scriptstyle $} (2);
    \draw[->] (2) to[bend left=6] node[pos=.5,auto] {$\scriptstyle $} (3);
    \draw[->] (1) to[bend right=6] node[pos=.5,auto,swap] {$\scriptstyle \psi_j^k$} (4);
    \draw[->] (3) to[bend left=6] node[pos=.5,auto] {$\scriptstyle $} (4);
  \end{tikzpicture}
\end{equation}
where
\begin{itemize}
\item[$-$] the upper left-hand map is $\id_{\Lambda^jF}\otimes \Delta$, with
  \begin{equation*}
    \Delta:S^kQ\to Q\otimes S^{k-1}Q
  \end{equation*}
  the partial comultiplication also appearing in \Cref{eq:3};
\item[$-$] the upper middle map is $\id\otimes\pi\otimes \id$, with $\pi:Q\to F$ the canonical surjection (again as in \Cref{eq:3});
\item[$-$] the upper right-hand map is the multiplication
  \begin{equation*}
    \Lambda^j F\otimes F\to \Lambda^{j+1}F
  \end{equation*}
  in the exterior algebra $\Lambda^{\bullet} F$, tensored with the identity on $S^{k-1}Q$. 
\end{itemize}

To show that the maps \Cref{eq:psij} fit into a resolution \Cref{eq:15} we begin with the following simple observation.

\begin{lemma}\label{le:jrows}
  Let $j,k$ be two positive integers and $X$ be a vector space of dimension larger than $j+1$. The map
\begin{equation}\label{eq:xmap}
  \begin{tikzpicture}[auto,baseline=(current  bounding  box.center)]
    \path[anchor=base] (0,0) node (1) {$\Lambda^j X\otimes S^{k}X$}
    +(4,1) node (2) {$\Lambda^jX\otimes X\otimes S^{k-1}X$}
    +(8,0) node (4) {$\Lambda^{j+1}X\otimes S^{k-1}X$};
    \draw[->] (1) to[bend left=6] node[pos=.5,auto] {$\scriptstyle \id_{\Lambda^j X}\otimes \Delta$} (2);
    \draw[->] (1) to[bend right=6] node[pos=.5,auto,swap] {$\scriptstyle $} (4);
    \draw[->] (2) to[bend left=6] node[pos=.5,auto] {$\scriptstyle \mathrm{mult}\otimes\id_{S^{k-1}X}$} (4);
  \end{tikzpicture}
\end{equation}
with $\Delta:S^k X\to X\otimes S^{k-1}X$ defined as in \Cref{eq:expldelta} annihilates the direct summand $\bS_{(k+1, (j-1)\times 1)} X$
 of $\Lambda^j X\otimes S^{k}X$ and maps the complementary summand  $\bS_{(k, j\times 1)} X$
of $\Lambda^j X\otimes S^{k}X$ isomorphically onto the corresponding summand of the codomain $\Lambda^{j+1} X\otimes S^{k-1}X${.}
\end{lemma}
\begin{proof}
  That the domain and codomain decompose as
  \begin{equation*}
    \Lambda^j X\otimes S^{k}X\cong \bS_{(k+1, (j-1)\times 1)}X \oplus \bS_{(k, j\times 1)}X
  \end{equation*}
  and
  \begin{equation*}
    \Lambda^{j+1} X\otimes S^{k-1}X\cong \bS_{(k, j\times 1)}X \oplus \bS_{(k-1, (j+1)\times 1)}X
  \end{equation*}
  respectively follows from the Littlewood-Richardson rule \cite[Appendix A, (A.8)]{fh-bk}. The claim can be checked on finite-dimensional vector spaces first, where all four direct summands are irreducible representations of the algebraic group $GL(X)$, then passing to arbitrary $X$ by taking a colimit.
\end{proof}

Now consider one of the objects $\Lambda^j F\otimes I$, $j\ge 0$ under discussion. Since $I$ has a filtration \Cref{eq:ifilt} with subquotients $S^k F$, $k\ge 0$, the object $\Lambda^j F\otimes I$ has a filtration with subquotients $\Lambda^j F\otimes S^k F$ decomposed as 
\begin{alignat*}{3}
  \Lambda^j F\otimes S^k F&\quad\cong\quad \bS_{(k+1, (j-1)\times 1)}F \quad \oplus\quad &&\bS_{(k, j\times 1)}F&&\quad \text{ for }j,k>0\\
  S^k F &\quad\cong\quad &&\bS_{(k)}F &&\quad \text{ for }j=0,\ k>0\numberthis\label{eq:decs}\\
  \Lambda^j F &\quad\cong\quad \bS_{(1,\ldots, 1)}F&& &&\quad \text{ for }k=0.
\end{alignat*}
Moreover, these decompositions are canonical, i.e. the summands are unique. 

We write
\begin{equation}\label{eq:kj}
  K_j:=\ker\left(\psi_j:\Lambda^j F\otimes I \to \Lambda^{j+1} F\otimes I\right)
\end{equation}
for the kernel of the map \Cref{eq:psij} and
\begin{equation*}
  K_j^k:= K_j\cap \left(\Lambda^j F\otimes S^kQ\right) = \ker\left(\psi_j:\Lambda^j F\otimes S^kQ \to \Lambda^{j+1} F\otimes S^{k-1}Q\right);
\end{equation*}
by convention, we set $K_j^{-1}=\{0\}$. We now have

\begin{lemma}\label{le:ksubq}
  For each $k\ge 0$ the subquotient
  \begin{equation}\label{eq:kfq}
    K_j^k/K_j^{k-1}\subset \Lambda^j F\otimes \left(S^kQ/S^{k-1}Q\right)\cong \Lambda^j F\otimes S^k F
  \end{equation}
  is the $j$-row summand of the latter. 
\end{lemma}
\begin{proof}
  The map $\psi_j$ respects the filtrations of its domain and codomain, by
  \begin{equation}\label{eq:filtrations}
    \Lambda^j F\otimes S^k Q\quad\text{and}\quad \Lambda^{j+1} F\otimes S^{k-1} Q
  \end{equation}
  respectively, and the associated graded map $\mathrm{gr}~\psi_j$, in degree $k$, is precisely \Cref{eq:xmap} with $X=F$. By \Cref{le:jrows} this means that the degree-$k$ kernel of $\mathrm{gr}~\psi_j$ is the $j$-row summand of $\Lambda^j F\otimes S^k F$. This verifies the statement at the associated-graded level.

  To conclude, it will suffice to construct gradings on the domain and codomain of $\psi_j$, compatible with $\psi_j$, that give back the filtrations by \Cref{eq:filtrations}. This would then prove that the filtered map $\psi_j$ arises from a grading, and hence that its kernel is the direct sum of the kernels of its homogeneous components. 

  We construct the requisite gradings as in the discussion preceding \Cref{le.correct-ker}: fix a basis $v_{\alpha}$ for $Q$ with $v_0=1\in \bC\subset Q$, and assign
  \begin{equation*}
    \deg v_{\alpha}=
    \begin{cases}
      0&\text{for }\alpha=0\\
      1&\text{otherwise}.
    \end{cases}
  \end{equation*}
  One checks easily that $\psi_j$ preserves degrees, finishing the proof as described above.
\end{proof}

We can now finally complete the discussion on the injective resolution \Cref{eq:15}.

\begin{theorem}\label{th:injresc}
  The morphisms \Cref{eq:psij} fit into an exact sequence \Cref{eq:15}. 
\end{theorem}
\begin{proof}
  The maps $\psi_j$ fit into a sequence 
  \begin{equation}\label{eq:targetseq}
    0\to \bC\to I\stackrel{\psi_0}{\longrightarrow} F\otimes I\stackrel{\psi_1}{\longrightarrow} \Lambda^2 F\otimes I\stackrel{\psi_2}{\longrightarrow} \cdots
  \end{equation}
  (not yet known to be exact) of filtered vector spaces. \Cref{le:jrows} applied to $X=F$ shows that the associated graded sequence is exact, and the conclusion follows from the fact that, as seen in the proof of \Cref{le:ksubq}, the filtrations on the terms of \Cref{eq:targetseq} arise from gradings compatible with the maps $\psi_j$. 
\end{proof}


\begin{corollary}\label{cor.c}
For a simple object $X$ of $\bT$ we have  

\begin{equation*}
\operatorname{Ext}_{\bT}^{j}(X,\bC)=
 \begin{cases}
  0& \textup{ if } X \not\simeq  L_{\lambda,\emptyset,\emptyset,{\lambda^\perp}} \textup{ for }\lambda \textup{ with } |\lambda|=j\\
\mathbb{C}& \textup{ if } X \simeq  L_{\lambda,\emptyset,\emptyset,{\lambda^\perp}}  \textup{ for }\lambda \textup{ with } |\lambda|=j.

 \end{cases}
\end{equation*}
\end{corollary}
\begin{proof}
  The statement follows from the existence of the injective resolution \Cref{eq:15} of the trivial object $\bC$, since by \Cref{thm:hulls}  \begin{equation*}
  \mathrm{soc}I_j=\mathrm{soc}(\Lambda^j F\otimes I)=\Lambda^j F = \Lambda^j(W_*\otimes W),
  \end{equation*}
 which in turn, by \Cref{pr.plth}, \Cref{item:2}, decomposes as
   \begin{equation*}
  \bigoplus_{|\lambda|=j}W_{*\lambda}\otimes W_{\lambda^{\perp}}. 
  \end{equation*}
  
\end{proof}

\subsection{Koszulity}\label{subse:kosz}

We will eventually show that the category $\bT$ is Koszul. To that end, we first need to strengthen \Cref{pr.def-ge} to an equality:

\begin{theorem}\label{thm:right-def}
  The Grothendieck category ${\bT}$ is sharp in the sense of \Cref{def.shrp}: for any two simples $S,T\in {\bT}$ we have
  \begin{equation*}
    \mathrm{Ext}^q(S,T)\ne 0\quad \Rightarrow\quad  d(s,t)=q.
  \end{equation*}
\end{theorem}

We do this in stages, considering the following particular case first. 

\begin{proposition}\label{pr.thk}
  \Cref{thm:right-def} holds when $T$ is purely thick. 
\end{proposition}
\begin{proof}
Let $T\in \cS_t$. We have to argue that there is some injective resolution
  \begin{equation*}
    0\to T\to K_0\to K_1\to \cdots
  \end{equation*}
  so that the socle of $K_q$ is a sum of simple objects $S$ with $d(S,T)=q$. This follows from \Cref{cor.c} for the trivial object $T=\bC$ and in general for thick $T$ from \Cref{thm:hulls}, which implies that we can obtain an injective resolution for $T$ by simply tensoring \Cref{eq:15} with $T$.
\end{proof}


In order to push past purely thick objects we need a version of \cite[Lemma 3.13]{us0}, requiring some notation: for a quadruple $(\lambda,\mu,\nu,\pi)$ we write $L^{+\ell}_{\lambda,\mu,\nu,\pi}$ for the direct sum of all $L_{\lambda,\mu',\nu,\pi}$ with $\mu'$ obtained by adding a box to $\mu$. Here $\ell$ stands for {\it left}, and we have a similarly defined object $L^{+r}_{\lambda,\mu,\nu,\pi}$ (for {\it right}) obtained by enlarging $\nu$ instead.

\begin{lemma}\label{le.exact}
  Consider a simple object $L_{\lambda,\mu,\nu,\pi}${.}
  \begin{enumerate}[(a)]
  \item\label{item:14} We have an exact sequence
    \begin{equation}\label{eq:18}
      0\to L^{+\ell}_{\lambda,\mu,\nu,\pi}\to V^*\otimes L_{\lambda,\mu,\nu,\pi} \to H \to 0
    \end{equation}
    where $H$ is a sum of simple objects $L_{\lambda',\mu,\nu,\pi}$ with $|\lambda'|=|\lambda|+1$ and $L_{\lambda,\mu,\nu',\pi}$ with $|\nu'|=|\nu|-1$. 
  \item\label{item:15} Tensoring with $V_*^*$ produces a similar exact sequence, containing $L_{\lambda,\mu,\nu,\pi}^{+r}$ rather than $L_{\lambda,\mu,\nu,\pi}^{+\ell}$. 
  \end{enumerate}
\end{lemma}
\begin{proof}
  We focus on \emph{(\labelcref{item:14})}, the other half being entirely analogous.

  The proof follows the same line of reasoning as that of \cite[Lemma 3.13]{us0}. We first tensor the extension
  \begin{equation*}
    0\to V_*\to V^*\to W_*\to 0
  \end{equation*}
  with $L_{\lambda,\mu,\nu,\pi}$ to obtain a sequence \Cref{eq:18} with an as yet unidentified $H$, itself fitting into an extension
  \begin{equation}\label{eq:19}
    0\to \widetilde{H} \to H \to W_*\otimes L_{\lambda,\mu,\nu,\pi}\to 0,
  \end{equation}
  with $\widetilde{H}$ being a direct sum of simples obtained by evaluating one tensorand $V_*$ against the $\nu$ component of $L_{\lambda,\mu,\nu,\pi}$. It follows that $\widetilde{H}$ is a direct sum of simple objects $L_{\lambda,\mu,\nu',\pi}$ with $|\nu'|=|\nu|-1$, and the splitting of \Cref{eq:19} follows from \Cref{pr.def-ge} and the observation that the indices of simple direct summands
   $L_{\lambda',\mu,\nu,\pi}$
    of $W_*\otimes L_{\lambda,\mu,\nu,\pi}$ and
    $L_{\lambda,\mu,\nu',\pi}$  of $\widetilde{H}$  are not comparable with respect to the partial order  \Cref{eq:ord}.
\end{proof}

\pf{thm:right-def}
\begin{thm:right-def}
This too presents no substantive difficulties beyond those encountered in \cite[Theorem 3.11]{us0}. Setting $s=:(\lambda,\mu,\nu,\pi)$ and $t=:(\lambda',\mu',\nu',\pi')$, the argument proceeds by induction on $q+|\mu'|+|\nu'|$. The base case follows from \Cref{pr.thk} for empty diagrams $\mu'$ and $\nu'$, and trivially for $q=0$. 

If $\mu'$ and $\nu'$ are empty we can fall back on \Cref{pr.thk}. Otherwise, suppose for instance that $\mu'$ is non-empty. We can then embed $T$ as a direct summand in $L^{+\ell}_{\lambda',\beta,\nu',\pi'}$ 
for $|\beta|=|\mu'|-1$. The assumed non-vanishing of $\mathrm{Ext}^q(S,T)$ and the long exact ext sequence applied to the extension
\begin{equation*}
  0\to L^{+\ell}_{\lambda',\beta,\nu',\pi'}\to V^*\otimes L_{\lambda',\beta,\nu',\pi'} \to H' \to 0
\end{equation*}
provided by \Cref{le.exact} forces us into one of two cases:

\vspace{.5cm}

{\bf 1: $\mathrm{Ext}^{q-1}(S,H')\ne 0.$} By the induction hypothesis we have $d(S,U)=q-1$ for some simple direct summand $U$ of $H'$, and the conclusion follows from this and the fact that $d(U,T)=1$ for all such $U$ (by the description of $H'$ in \Cref{le.exact} and the formula for the defect provided by \Cref{le.defct}). 

\vspace{.5cm}

{\bf 2: $\mathrm{Ext}^{q}(S,V^*\otimes L_{\lambda',\beta,\nu',\pi'})\ne 0.$} This means that $S$ is a direct summand of the socle of $Z_q$ for any injective resolution
\begin{equation}\label{eq:20}
  0\to V^*\otimes L_{\lambda',\beta,\nu',\pi'}\to Z_0\to Z_1\cdots.
\end{equation}
By induction we know that the socle of the $q^{th}$ term of an injective resolution
\begin{equation}\label{eq:21}
  0\to L_{\lambda',\beta,\nu',\pi'}\to Y_0\to Y_1\cdots
\end{equation}
consists of simples $U$ with $d(U,L_{\lambda',\beta,\nu',\pi'})=q$. Now note that an injective resolution \Cref{eq:20} can be obtained by tensoring \Cref{eq:21} with $V^*$. The simple direct summands of the socle of $V^*\otimes Y_q$, including $S$, differ from those of $Y_q$ in that their $\mu$ diagrams have one extra box, meaning that indeed
\begin{equation*}
  d(S,T) = d(U,L_{\lambda',\beta,\nu',\pi'})=q.
\end{equation*}

\vspace{.5cm}

The case when $\mu'$ is empty but $\nu'$ is not proceeds analogously, making use of part \emph{(\labelcref{item:15})} of \Cref{le.exact} rather than \emph{(\labelcref{item:14})}. 
\end{thm:right-def}

As a direct consequence of \Cref{thm:kosz-gen,thm:right-def} we have

\begin{theorem}\label{thm:4kosz}
  The ordered Grothendieck category ${\bT}$ is Koszul in the sense of \Cref{def.ksz}.
  \qedhere
\end{theorem}

Furthermore, we have the following analogue of \cite[Corollary 3.19]{us0} and \cite[Corollary 4.25 (d)]{us1}. In the statement, $\bT_{fin}\subset \bT$ denotes the full subcategory consisting of finite-length objects.

\begin{theorem}\label{th:koszcoalg}
  The Grothendieck category ${\bT}$ is equivalent to the category $\cM^C$ of comodules over a Koszul graded coalgebra $C$, with $\bT_{fin}\simeq \cM^C_{fin}$.
\end{theorem}
\begin{proof}
  The hypotheses of \Cref{thm:tak} are met (for the ground field $\bC$): ${\bT}$ is generated by the finite-length objects in $\bT$, since every object is isomorphic to a subquotient of a direct sum of indecomposable injectives $I_s$ as defined in \Cref{eq:3}, and in turn the injectives $I_s$ are unions of their finite-length truncations
  \begin{equation*}
    J_s\otimes S^kQ\text{ for } k\in \bZ_{>0}. 
  \end{equation*}
  Moreover, according to \Cref{th:scalaralgs}, the endomorphism ring of a simple object $L_{\lambda,\mu,\nu,\pi}$ as in \Cref{eq:s} is the field $\bC$.
\end{proof}

\subsection{An internal commutative algebra and its modules}\label{subse:intalg}

The object $I$ has a structure of commutative algebra internal to the tensor category ${}_{\fg}\md$ of $\fg$-modules. To see this, we observe that $I$ is isomorphic as a $\fg$-module to a quotient algebra of the symmetric algebra $S^{\bullet}Q$.
 Indeed, denote by $a$ the distinguished element $1\in \bC\subset Q$ of the degree-one component 
  $Q\subset S^{\bullet}Q$
and consider the commutative algebra 
\begin{equation*}
  S^{\bullet}Q/(a-1),
\end{equation*}
where 1 ist the unit of the symmetric algebra and $(a-1)$ is the ideal generated by $a-1$. This ideal is clearly $\fg$-stable and $S^{\bullet}Q/(a-1)$ is an algebra in ${}_\fg \md$. Moreover, the definition of $I$ as $\varinjlim S^k Q$ implies that there is an isomorphism of $\fg$-modules \begin{equation*}
I\cong S^{\bullet}Q/(a-1).
\end{equation*}

	
	
We will be interested in the category ${}_I\bT$ of $I$-modules internal to $\bT$. This is clearly a Grothendieck category. Moreover, the forgetful functor
\begin{equation*}
  \cat{forget}:{}_I\bT\to \bT
\end{equation*}
fits into an adjunction
\begin{equation}\label{eq:tit}
 \begin{tikzpicture}[auto,baseline=(current  bounding  box.center)]
   \path[anchor=base] (0,0) node (1) {$\bT$} 
   +(4,0) node (2) {${}_I\bT.$}
   +(2,0) node (3) {$\bot$}
   ;
   \draw[->] (1) to[bend left=16] node[pos=.5,auto] {$\scriptstyle I\otimes\bullet$} (2);
   \draw[->] (2) to[bend left=16] node[pos=.5,auto] {$\scriptstyle \cat{forget}$} (1);
 \end{tikzpicture}
\end{equation}
We refer to $I$-modules in the image of $I\otimes \bullet$ as {\it free}. We will see that tensoring with $I$ has the effect of ``partially semisimplifying'' $\bT$, in the following sense.

\begin{proposition}\label{pr:ss}
  For every positive integer $n$, the filtration
  \begin{equation*}
    0\subset \bC\subset Q\subset S^2Q\subset\cdots\subset S^nQ
  \end{equation*}
  splits in ${}_I\bT$ upon tensoring it with $I$. Consequently, for every $n$ the object $I\otimes S^nQ$ is injective in $\bT$. 
\end{proposition}
\begin{proof}
  We prove this inductively on $n$, starting with $n=1$. In this case the claim is that the embedding
  \begin{equation*}
    I\cong I\otimes \bC\subset I\otimes Q
  \end{equation*}
  splits in ${}_I\bT$. To see this, consider the embedding
  \begin{equation*}
    Q\to \cat{forget}~I
  \end{equation*}
  in $\bT$. It corresponds, via the adjunction \Cref{eq:tit}, to a morphism in $_I\bT$
  \begin{equation*}
    {\sigma:}\ I\otimes Q\to I
  \end{equation*}
  that is clearly the identity on the $I$-submodule
  \begin{equation*}
    I\cong I\otimes \bC\subset I\otimes Q.
  \end{equation*}
  The morphism $\sigma$ is the required splitting, concluding the base case $n=1$ of the induction.

  The argument also shows that we have a decomposition
  \begin{equation}\label{eq:iqiif}
    I\otimes Q\cong I\oplus (I\otimes F) 
  \end{equation}
  in ${}_I\bT$ (and hence also in $\bT$), implying that $I\otimes Q\in \bT$ is injective (by \Cref{thm:hulls} for instance, which shows that both summands in \Cref{eq:iqiif} are injective).

  We regard $Q$ as a subobject of $S^2Q$ via the embedding $Q\hookrightarrow S^2Q$ described in \Cref{eq:sksk1}. By the injectivity of $I\otimes Q\in \bT$ noted above, the embedding
  \begin{equation*}
    Q\cong \bC\otimes Q\subset I\otimes Q
  \end{equation*}
  extends to a morphism in $\bT$
  \begin{equation*}
    S^2Q\to I\otimes Q. 
  \end{equation*}
  Once more, the adjunction \Cref{eq:tit} retrieves a morphism in $_I\bT$
  \begin{equation*}
    I\otimes S^2Q\to I\otimes Q
  \end{equation*}
  that restricts to the identity on the submodule
  \begin{equation*}
    I\otimes Q\subset I\otimes S^2 Q,
  \end{equation*}
  showing that this embedding splits in ${}_I\bT$. This proves the main claim for $n=2$ and the fact that there is a splitting
    \begin{equation*}
      I\otimes S^2Q\cong \left(I\otimes Q\right)\oplus \left(I\otimes S^2F\right),
    \end{equation*}
    meaning that $I\otimes S^2Q$ is injective in $\bT$. We now repeat the procedure recursively to complete the inductive argument.
\end{proof}



Since it will be our goal to study the category ${}_I\bT$ along the same lines as $\bT$, we next turn to simple objects therein.

\begin{theorem}\label{th:itsimp}
  The simple objects in ${}_I\bT$ are (up to isomorphism) precisely the free $I$-modules $I\otimes S$ for simples $S\in \bT$. For each of them, the endomorphism algebra in ${}_I\bT$ is $\bC$. 
\end{theorem}
\begin{proof}
  We first prove that $I\otimes S$ is simple in ${}_I\bT$. The simple objects of $\bT$ are precisely the various modules $L_{\lambda,\mu,\nu,\pi}$ of \Cref{eq:s}, and according to \Cref{thm:hulls} the injective hull $S\subset I_S$ contains $I\otimes S$ ($I_S$ exists because ${}_I \bT$ is a Grothendieck category). Since
  \begin{equation*}
    S\cong \bC\otimes S\subset I\otimes S
  \end{equation*}
  is essential in $I_S$, it is also essential in $I\otimes S$. It follows that any non-zero subobject
  \begin{equation*}
    T\subset I\otimes S
  \end{equation*}
  in ${}_I\bT$ contains $S$ and hence the $I$-module $I\otimes S$ it generates, so $T=I\otimes S$. This concludes the proof of the claim that all $I\otimes S$ are simple.

  To see that
  \begin{equation}\label{eq:isomendiotimess}
    \End_{{}_I\bT}(I\otimes S)\cong \bC,
  \end{equation}
  note first that by the adjunction \Cref{eq:tit} we have
  \begin{equation*}
    \End_{{}_I\bT}(I\otimes S) \cong \Hom_{\bT}(S, I\otimes S). 
  \end{equation*}
  The claim that this must be $\bC$ follows from the fact that $I\otimes S$ is the injective hull of $S$ (\Cref{thm:hulls}) and hence
  \begin{itemize}
  \item[$-$] every morphism $S\to I\otimes S$ in $\bT$ factors through the socle $S\subset I\otimes S$, and
  \item[$-$] $\End_{\bT} S\cong \bC$. 
  \end{itemize}
Since $I\otimes S$ is an injective hull of $S$ in $\bT$, every morphism $S\rightarrow I\otimes S$ factors through the socle $S\subset I\otimes S$, the isomorphism $\End_{\bT} S=\bC$ implies the existence of an isomorphism (\ref{eq:isomendiotimess}).

As for the fact that these are, up to isomorphism, all irreducible objects in ${}_I\bT$, consider such an object $T$ and note that it must contain some simple $S\in \bT$. Hence $T$ must be a quotient of the (irreducible!) free object $I\otimes S\in {}_I\bT$. Consequently, we have $T\cong I\otimes S$ as desired.
\end{proof}

Since $I$ is a {\it commutative} algebra in $\bT$, the category ${}_I\bT$ of internal modules has a natural symmetric monoidal structure with $I$ as the unit object and $\otimes_I$ as the tensor product. Whenever we refer to ${}_I\bT$ as a tensor category, this will be the structure we consider. 

\subsection{The category ${}_I\bT$}\label{subse:modsord}

We are now ready to apply to ${}_I\bT$ the same treatment we subjected $\bT$ to. We work with precisely the same poset $(\cP,\preceq)$ of quadruples $(l,m,n,p)$ of non-negative integers with the ordering described in \Cref{eq:ord}, and the corresponding objects $I_s=I\otimes J_s\in {}_I\bT$ for $s\in \cP$, as in equations \Cref{eq:s} to \Cref{eq:13}.  

We will similarly consider the (simple, by \Cref{th:itsimp}) objects of ${}_I\bT$
\begin{equation}\label{eq:deft}
  T_{\lambda,\mu,\nu,\pi}:=I\otimes L_{\lambda,\mu,\nu,\pi}
\end{equation}
and the semisimple objects $T_{l,m,n,p}:=I\otimes L_{l,m,n,p}$, 
 that are direct sums of the various $T_{\lambda,\mu,\nu,\pi}$. 

We now have the following analogue of \Cref{pr.is-ordered}.

\begin{proposition}\label{pr.itis-ordered}
${}_I\bT$ is an ordered Grothendieck category in the sense of \Cref{def.ord-gr}.   
\end{proposition}
\begin{proof}
  Taking as above the objects $X_s$ to be our $I_s$ (this time regarded as objects in ${}_I\bT$ rather than just $\bT$), the argument proceeds much as in the proof of \Cref{pr.is-ordered} with a small difference in how we define the morphisms $I_s\to I_t$, $t\prec s$ in how we define the morphism from \Cref{def.ord-gr}, (\labelcref{item:9}).

  Once again, said morphisms will be tensor products and compositions of a few building blocks:
  \begin{itemize}
  \item[$-$] projecting one of the tensorands $V^*$ of $I_s=J_s\otimes I$ onto $W_*$;
  \item[$-$] the dual analogue, $V^*_*\to W$;
  \item[$-$] the ``pairing''
    \begin{equation}\label{eq:realpair}
      I_{0,1,0,0}\otimes_I I_{0,0,1,0} = (I\otimes V^*)\otimes_{I} (I\otimes V_*^*)\cong I\otimes V^*\otimes V_*^*\to I
    \end{equation}
    obtained via the adjunction \Cref{eq:tit} from the composition
    \begin{equation*}
      V^*\otimes V_*^*\to Q\subset I
    \end{equation*}
    in \Cref{eq:partpair}. 
  \end{itemize}
  Everything else goes through as sketched in the proof of \Cref{pr.is-ordered}. 
\end{proof}

The difference to $\bT$ is that now the free $I$-modules generated by the full duals $V^*$ and $V_*^*$ admit the pairing \Cref{eq:realpair} valued in the {\it unit object} $I$ of the category ${}_I\bT$ under consideration.

We also have an $I$-module version of \Cref{thm:hulls}. 

\begin{theorem}\label{thm:ihulls}
    For every quadruple $(\lambda,\mu,\nu,\pi)$ of Young diagrams the inclusion
  \begin{equation*}
    T_{\lambda,\mu,\nu,\pi} = I\otimes L_{\lambda,\mu,\nu,\pi} \subseteq I_{\lambda,\mu,\nu,\pi} = I\otimes J_{\lambda,\mu,\nu,\pi}
  \end{equation*}
  obtained by applying the functor $I\otimes \bullet$ to the inclusion
  \begin{equation*}
    L_{\lambda,\mu,\nu,\pi} \subseteq J_{\lambda,\mu,\nu,\pi}
  \end{equation*}
  is an injective hull in ${}_I\bT$.
  \qedhere
\end{theorem}

Just as $\bT$, the category ${}_I\bT$ can be realized as comodules over a coalgebra (see \Cref{th:koszcoalg}). As in that previous result, we denote by ${}_I\bT_{fin}\subset {}_I\bT$ the full subcategory of finite-length objects. Note that the indecomposable injectives
\begin{equation*}
I_{\lambda,\mu,\nu,\pi} = I\otimes J_{\lambda,\mu,\nu,\pi}  \in {}_I\bT
\end{equation*}
have finite length: $J_{\lambda,\mu,\nu,\pi}$ have finite filtrations with subquotients simple in $\bT$, and according to \Cref{th:itsimp} tensoring these simple objects by $I$ produces simples in ${}_I\bT$.

\begin{theorem}\label{th:ikoszcoalg}
  The Grothendieck category ${}_I{\bT}$ is equivalent to the category $\cM^D$ of comodules over a coalgebra $D$, with ${}_I\bT_{fin}\simeq \cM^D_{fin}$. Furthermore, the coalgebra $D$ is left semiperfect in the sense of \Cref{def:perf}. 
\end{theorem}
\begin{proof}
  The argument is largely parallel to that underpinning \Cref{th:koszcoalg}, via \Cref{thm:tak} (minus Koszulity, which we have not yet addressed for $I$-modules).

  The additional remark, that $D$ is semiperfect, follows directly from \Cref{def:perf} and the fact that, as observed above, in ${}_I\bT$ the indecomposable injectives $I_{\lambda,\mu,\nu,\pi}$ have finite length.
\end{proof}

We also need the following remark, which parallels \cite[Lemma 2.19]{us0} (the proof is virtually identical, so we omit it).

\begin{lemma}\label{le:fullonx}
  The tensor subcategory ${}_I\bT'$ of ${}_I\bT$ generated by the morphisms described in the proof of \Cref{pr.itis-ordered} is the full subcategory containing $I_{l,m,n,p}$.  \qedhere
\end{lemma}

We next turn to the Koszulity of ${}_I\bT$. In keeping with the theme, the argument will be very similar to what we saw in proving \Cref{thm:right-def,thm:4kosz}. We collect all of the statements together as follows.

\begin{theorem}\label{th:ikosz}
  The Grothendieck category ${}_I\bT$ is sharp in the sense of \Cref{def.shrp}: for any two simples $S,T\in {\bT}$ we have
  \begin{equation*}
    \mathrm{Ext}^q(S,T)\ne 0\quad \Rightarrow\quad  d(s,t)=q.
  \end{equation*}
  In particular, the ordered Grothendieck category ${}_I\bT$ is Koszul in the sense of \Cref{def.ksz}.
\end{theorem}
\begin{proof}
The last claim follows from sharpness by \Cref{thm:kosz-gen}, so we focus on proving the sharpness claim. In turn, the latter follows as in the proof of \Cref{thm:right-def}, with the exact sequence \Cref{eq:18} replaced by its analogue, obtained by simply tensoring it with $I$.   
\end{proof}

\begin{remark}
  Note that in the present setting the proof of Koszulity is in fact simpler than in \Cref{subse:kosz}: we do not need a version of \Cref{pr.thk}, since for purely thick simple objects $L_{\lambda,\emptyset,\emptyset,\pi}$ the corresponding simple object of ${}_{I}\bT$
  \begin{equation*}
    T_{\lambda,\emptyset,\emptyset,\pi} = I\otimes L_{\lambda,\emptyset,\emptyset,\pi}
  \end{equation*}
  is injective. 
\end{remark}

As a consequence, we can supplement \Cref{th:ikoszcoalg}, fully bringing it in line with \Cref{th:koszcoalg}. 

\begin{corollary}\label{cor.dgradedkoszul}
The coalgebra $D$ in \Cref{th:ikoszcoalg} can be chosen graded and Koszul. 
  \qedhere
\end{corollary}

We end the present subsection with description of one possible choice for the graded coalgebra $C$ in \Cref{th:koszcoalg}. This discussion parallels \cite[\S 3.4]{us0}, which in turn is analogous to \cite[\S 5]{DPS}.

Let $T$ be the tensor algebra in ${}_I\bT$ of the object
\begin{equation*}
  (W_*\otimes I)\oplus (V^*\otimes I)\oplus (V_*^*\otimes I)\oplus (W\otimes I)  = I_{1,0,0,0}\oplus I_{0,1,0,0}\oplus I_{0,0,1,0}\oplus I_{0,0,0,1}
\end{equation*}
with $T_d$ denoting its degree-$d$ component, and the non-unital algebra of endomorphisms
\begin{equation*}
  \mathcal{A}:=\bigoplus_{m,n\in \bZ_{\ge 0}}\Hom_{{}_I\bT}(T_m,T_n)\label{eq:a}=\bigoplus_{s,t\in \cP}\Hom_{{}_I\bT}(I_s, I_t).
\end{equation*}
The algebra $\mathcal{A}$ is naturally $\bZ_{\ge 0}$-graded by means of the defect introduced in \Cref{def:defect}:
\begin{equation*}
  \mathcal{A}_d:=\bigoplus_{\substack{s,t\in \cP \\ d(t,s)=d}}\Hom_{{}_I\bT}(I_s, I_t).
\end{equation*}
Finally, the coalgebra $C$ is simply the graded dual of $\mathcal{A}$, with $C_d=\mathcal{A}_d^*$.

The fact that $C$ (and hence $\mathcal{A}$) is Koszul then implies

\begin{proposition}\label{pr:isquad}
  The algebra $\mathcal{A}$ is quadratic. 
\end{proposition}
\begin{proof}
  Koszul algebras are well known to be quadratic; see e.g. \cite[\S 2.3]{BGS}.
\end{proof}

\subsection{Universality}\label{subse:univ}

We can now characterize ${}_I\bT$ as a universal category in the sense of \cite[Theorem 4.23]{us0} and \cite[Theorem 5.2]{us1}. First, note that in ${}_I\bT$ there is a pairing
\begin{equation}\label{eq:pairi}
  (I\otimes V^*)\otimes_I (I\otimes V_*^*)\cong I\otimes V^*\otimes V_*^*\to I
\end{equation}
corresponding to \Cref{eq:partpair} through the adjunction \Cref{eq:tit}. 

\begin{theorem}\label{th:sl-univ}
  Let $(\cD,\ \otimes,\ {\bf 1})$ be a ($\bC$-linear abelian) tensor category, $x\hookrightarrow x_*^*$ and $x_*\hookrightarrow x^*$ be monomorphisms in $\cD$, and
  \begin{equation*}
    p: x^*\otimes x_*^*\to {\bf 1}
  \end{equation*}
  be a morphism in $\cD$.

  \begin{enumerate}[(a)]
  \item There is a unique (up to monoidal natural isomorphism) left exact symmetric monoidal functor $F:{}_I\bT_{fin}\to \cD$ sending
    \begin{itemize}
    \item[$-$] the pairing \Cref{eq:pairi} to $p$;
    \item[$-$] the surjection $V_*^*\to W$ to $x_*^*\to x_*^*/x$;
    \item[$-$] the surjection $V^*\to W_*$ to $x^*\to x_*$. 
    \end{itemize}
  \item if $\cD$ is additionally a Grothendieck category then $F$ extends uniquely to a coproduct-preserving functor ${}_I\bT\to \cD$.
  \end{enumerate}
\end{theorem}

The argument will be analogous to that employed in the proof of \cite[Theorem 3.23]{us0}, revolving around the fact that the algebra $A$ in the preceding discussion is quadratic (\Cref{pr:isquad}). For that reason, it will be necessary to understand its components of degree $\le 2$. In degree zero things are simple: the following result is the version of \cite[Lemma 3.24]{us0} appropriate here.

\begin{lemma}\label{le:deg0}
  For $(l,m,n,p)\in \cP$ the endomorphism algebra of the injective object $I_{l,m,n,p}\in {}_I\bT$ is isomorphic to $\bC[S_l\times S_m\times S_n\times S_p]$, with the symmetric groups acting naturally on the relevant tensorands of
  \begin{align*}
    I_{l,m,n,p} &= I\otimes J_{l,m,n,p}\\
                &=I\otimes W_*^{\otimes l}\otimes (V^*)^{\otimes m}\otimes (V_*^*)^{\otimes n} \otimes W^{\otimes p}.
  \end{align*}
\end{lemma}
\begin{proof}
  We have
  \begin{equation*}
    \End_{{}_I\bT}(I\otimes J_{l,m,n,p})\cong \Hom_{\bT}(J_{l,m,n,p},I\otimes J_{l,m,n,p}). 
  \end{equation*}
  The quotient of $J_{l,m,n,p}$ by its socle in $\bT$ has a filtration by subquotients $L_{\lambda',\mu',\nu',\pi'}$ with
  \begin{equation*}
    (|\lambda'|,|\mu'|,|\nu'|,|\pi'|) \prec (l,m,n,p)\in \cP,
  \end{equation*}
  which thus admit no non-zero morphisms to
  \begin{equation*}
    \mathrm{soc}(I\otimes J_{l,m,n,p} )= \mathrm{soc}J_{l,m,n,p}. 
  \end{equation*}
  It follows that the restricting an arbitrary morphism 
  \begin{equation*}
    J_{l,m,n,p}\to I\otimes J_{l,m,n,p}
  \end{equation*}
  in $\bT$ to the socle induces an isomorphism
  \begin{equation}\label{eq:homend}
    \Hom_{\bT}(J_{l,m,n,p},I\otimes J_{l,m,n,p})
    \cong
    \End_{\bT}\left(W_*^{\otimes l}\otimes (V^*)^{\otimes m}\otimes (V_*^*)^{\otimes n} \otimes W^{\otimes p}\right). 
  \end{equation}
  We can see that the right-hand-side of \Cref{eq:homend} is naturally identifiable with $\bC[S_l\times S_m\times S_n\times S_p]$ as in \Cref{pr:gpendoalgs}.
\end{proof}

As for degree $1$, we need an analogue of \cite[Lemma 3.25]{us0}. Stating such an analogue will require some notation. Degree-one morphisms between the objects $I_{l,m,n,p}\in {}_I\bT$ come in three flavors:
\begin{align*}
  I_{l,m,n,p}&\to I_{l,m-1,n-1,p},\\
  I_{l,m,n,p}&\to I_{l+1,m-1,n,p},\\
  I_{l,m,n,p}&\to I_{l,m,n-1,p+1}.
\end{align*}
We distinguish families of each flavor, as follows. The morphism
\begin{equation*}
  \phi_{i,j}:I_{l,m,n,p}\to I_{l,m-1,n-1,p}\ \text{ for } 1\le i \le m,~1\le j \le n
\end{equation*}
applies the pairing
\begin{equation*}
  V^*\otimes V_*^*\to I
\end{equation*}
on the $i^{th}$ tensorand $V^*$ and the $j^{th}$ tensorand $V_*^*$ of
\begin{equation*}
  I_{l,m,n,p} = I\otimes W_*^{\otimes l}\otimes (V^*)^{\otimes m}\otimes (V_*^*)^{\otimes n} \otimes W^{\otimes p}
\end{equation*}
and acts as the identity on all other tensorands.

Next, we have the maps
\begin{equation*}
{}_{i,j}\pi:I_{l,m,n,p}\to I_{l+1,m-1,n,p}\ \text{ for } 1\le i \le m,~1\le j \le n
\end{equation*}
which
\begin{itemize}
\item[$-$] first permutes cyclically the first $i$ tensorands $V^*$;
\item[$-$] maps the new first (old $i^{th}$) tensorand $V^*$ onto $W_*=V^*/V_*$;
\item[$-$] finally permutes the last $m-j+1$ tensorands $W^*$ cyclically so the newly-created $W_*$ becomes the $j^{th}$.
\end{itemize}

Finally, we have the left-right mirror image 
\begin{equation*}
\pi_{i,j}:I_{l,m,n,p}\to I_{l,m,n-1,p+1}\ \text{ for } 1\le i \le m,~1\le j \le n
\end{equation*}
of ${}_{i,j}\pi$, obtained by substituting $V_*^*$ for $V^*$, $W$ for $W_*$, reversing the directions of the cyclic permutations, etc.

We write
\begin{equation*}
  S_{l,m,n,p}:=S_l\times S_m\times S_n\times S_p
\end{equation*}
for products of symmetric groups and unless specified otherwise, morphism spaces in \Cref{le:deg1} below are in the category ${}_I\bT$. 

\begin{lemma}\label{le:deg1}
  Let $(l,m,n,p)\in \cP$.
  \begin{enumerate}[(a)]
  \item\label{item:16} $\Hom(I_{l,m,n,p},I_{l,m-1,n-1,p})$ is isomorphic to $\bC[S_{l,m,n,p}]$ as a bimodule over
    \begin{equation*}
      \End I_{l,m-1,n-1,p} \cong \bC [S_{l,m-1,n-1,p} ]\text{ and }\End I_{l,m,n,p}\cong \bC [S_{l,m,n,p}],
    \end{equation*}
    with any of the morphisms $\phi_{i,j}$ as a generator for the right $\bC [S_{l,m,n,p} ]$-module structure while identifing the subgroups
    \begin{equation*}
     S_{m-1}\subset S_{m} \text{ and } S_{n-1} \subset S_{n}
    \end{equation*}
     with the isotropy groups of $i$ and $j$ respectively. 
  \item\label{item:17} $\Hom(I_{l,m,n,p},I_{l+1,m-1,n,p})$ is isomorphic to the induced module
    \begin{equation*}
      \bC [S_{l+1,m,n,p}]\cong \mathrm{Ind}_{S_l}^{S_{l+1}} \bC [S_{l,m,n,p}] = \bC [S_{l+1}]\otimes_{\bC [S_l]}\bC [S_{l,m,n,p}]
    \end{equation*}
    as a bimodule over
    \begin{equation*}
      \End I_{l+1,m-1,n,p}\cong \bC [S_{l+1,m-1,n,p}]\text{ and }\End I_{l,m,n,p} \cong \bC [S_{l,m,n,p}],
    \end{equation*}
    with ${}_{i,j}\pi$ as a generator for the right $\bC [S_{l,m,n,p}]$-module structure while identifying the subgroups
    \begin{equation*}
      S_l\subset S_{l+1}\text{ and } S_{m-1}\subset S_m
    \end{equation*}
    with the isotropy groups of $j$ and $i$ respectively. 
  \item\label{item:18} The left-right mirror image of \Cref{item:17}: $\Hom(I_{l,m,n,p},I_{l,m,n-1,p+1})$ is isomorphic to the induced module
    \begin{equation*}
      \bC [S_{l,m,n,p+1}]\cong \mathrm{Ind}_{S_p}^{S_{p+1}} \bC [S_{l,m,n,p}] = \bC [S_{p+1}]\otimes_{\bC [S_p]}\bC [S_{l,m,n,p}]
    \end{equation*}
    as a bimodule over
    \begin{equation*}
      \End I_{l,m,n-1,p+1}\cong \bC [S_{l,m,n-1,p+1}]\text{ and }\End I_{l,m,n,p}\cong \bC [S_{l,m,n,p}],
    \end{equation*}
    with $\pi_{i,j}$ as a generator for the right $\bC [S_{l,m,n,p}]$-module structure while identifying the subgroups
    \begin{equation*}
      S_p\subset S_{p+1}\text{ and } S_{n-1}\subset S_n
    \end{equation*}
    with the isotropy groups of $j$ and $i$ respectively. 
  \end{enumerate}
\end{lemma}
\begin{proof}  
  {\Cref{item:16}} Having fixed $i$ and $j$ as in the statement, we have a morphism
  \begin{equation}\label{eq:mnm1n1}
    \bC [S_{l,m,n,p}]\to \Hom(I_{l,m,n,p},I_{l,m-1,n-1,p})
  \end{equation}
  of $(\bC[S_{l,m-1,n-1,p}], \bC[S_{l,m,n,p}])$-bimodules, sending $1$ to $\phi_{i,j}$. The fact that that morphism is surjective follows from \Cref{le:fullonx}, so it is the injectivity that we focus on.   

  Note that the morphism \Cref{eq:mnm1n1} factors as
  \begin{equation}\label{eq:largediag}
    \begin{tikzpicture}[auto,baseline=(current  bounding  box.center)]
      \path[anchor=base] 
      (0,0) node (l) {$\bC [S_{l,m,n,p}]$}
      +(1,2) node (ul) {$\bC [S_{l,p}]\otimes \bC [S_{m,n}]$}
      +(8,2) node (ur) {$(\End I_{l,0,0,p})\otimes \Hom(I_{0,m,n,0},I_{0,m-1,n-1,0})$}
      +(9,0) node (r) {$\Hom(I_{l,m,n,p},I_{l,m-1,n-1,p})$}
      ;
      \draw[->] (l) to[bend left=6] node[pos=.5,auto] {$\scriptstyle \cong$} (ul);
      \draw[->] (ul) to[bend left=0] node[pos=.5,auto] {$\scriptstyle $} (ur);
      \draw[->] (ur) to[bend left=6] node[pos=.5,auto,swap] {$\scriptstyle $} (r);
      \draw[->] (l) to[bend right=0] node[pos=.5,auto,swap] {$\scriptstyle $} (r);
    \end{tikzpicture}
  \end{equation}
  where the downward arrow is the tensor product (over the unit object $I\in {}_I\bT$) of morphisms in ${}_I\bT$. Since the vertical maps are injective, the bottom morphism (which is our target) will be one-to-one if and only if the top arrow is. The left-hand tensorand
  \begin{equation*}
    \bC [S_{l,p}]\to \End I_{l,0,0,p}
  \end{equation*}
  of the top map in \Cref{eq:largediag} is an isomorphism by \Cref{le:deg0}, so it is enough to consider the right hand tensorand
  \begin{equation*}
    \bC [S_{m,n}]\to \Hom(I_{0,m,n,0},I_{0,m-1,n-1,0})
  \end{equation*}
  of that map; equivalently, it suffices to resolve the present discussion in the case $l=p=0$. But this follows from \cite[Lemma 3.25 (a)]{us0}, analogous to the result being proven here, by noting that the restrictions of the compositions
  \begin{equation*}
    \phi_{i,j}\circ \sigma: I_{0,m,n,0}\to I_{0,m-1,n-1,0},\quad \sigma\in S_{m,n}
  \end{equation*}
  to
  \begin{equation*}
    (V^*)^{\otimes m}\otimes V^{\otimes n}\subset (V^*)^{\otimes m}\otimes (V_*^*)^{\otimes n}\subset I_{0,m,n,0}
  \end{equation*}
  are precisely the morphisms proven linearly independent there. 

  {\Cref{item:17}} We again have an $(\bC[S_{l+1,m-1,n,p}],\bC[S_{l,m,n,p}])$-bimodule map
  \begin{equation}\label{eq:lml1m1}
    \bC [S_{l+1,m,n,p}] \to \Hom(I_{l,m,n,p},I_{l+1,m-1,n,p})
  \end{equation}
  sending $1$ to ${}_{i,j}\pi$, and its surjectivity is a consequence of \Cref{le:fullonx}. The injectivity follows as in part \Cref{item:16}, by first decomposing \Cref{eq:lml1m1} as a tensor product of maps
  \begin{equation*}
    \bC [S_{l+1,m}] \to \Hom(I_{l,m,0,0},I_{l+1,m-1,0,0})
  \end{equation*}
  and
  \begin{equation*}
    \bC [S_{n,p}] \to \End I_{0,0,n,p}. 
  \end{equation*}
  The latter is an isomorphism by \Cref{le:deg0}, and the former is an injection as in the proof of \Cref{item:16}, by appealing to \cite[Lemma 3.25 (b)]{us0}. 
  
  {\Cref{item:18}} As noted in the statement, this is entirely parallel to part \Cref{item:17}, interchanging the roles of $V^*$ and $V_*^*$, $l$ and $p$, $m$ and $n$, etc.
\end{proof}


The composition map from the degree-one to the degree-two component of $A$ comes in several varieties, depending on the domain. Before listing the various options, it will be convenient to introduce

\begin{notation}\label{not:iud}
  For a quadruple $(l,m,n,p)\in \cP$ we write
  \begin{equation*}
    I_{l,m\downarrow,n\downarrow,p}:=\Hom_{{}_I\bT}(I_{l,m,n,p},I_{l,m-1,n-1,p})
  \end{equation*}
  and similarly for other morphism spaces, with arrows indicating whether the respective index increases or decreases, and multiple arrows to indicate the amount. Other examples are
  \begin{align*}
    I_{l,m\downarrow\downarrow,n\downarrow\downarrow,p}&:=\Hom_{{}_I\bT}(I_{l,m,n,p},I_{l,m-2,n-2,p}){,}\\
    I_{l\uparrow,m\downarrow,n,p}&:=\Hom_{{}_I\bT}(I_{l,m,n,p},I_{l+1,m-1,n-2,p})\\
  \end{align*}
  and so on.

  For composable morphism spaces in ${}_I\bT$ we denote by `$\odot$' the tensor product over the endomorphism algebra of the intermediate object. For example:
  \begin{equation*}
    I_{l,(m-1)\downarrow,(n-1)\downarrow,p}\odot I_{l,m\downarrow,n\downarrow,p}
    :=
    I_{l,(m-1)\downarrow,(n-1)\downarrow,p}\otimes_{\bC [S_{l,m-1,n-1,p}]} I_{l,m\downarrow,n\downarrow,p}{.}
  \end{equation*}
\end{notation}


With in place, the possibilities for composition of degree-1 morphisms are:

\begin{equation}\label{eq:22}
  I_{l,(m-1)\downarrow,(n-1)\downarrow,p}\odot I_{l,m\downarrow,n\downarrow,p}
  \to
  I_{l,m\downarrow\downarrow,n\downarrow\downarrow,p},
\end{equation}

which is mirror-self-dual,

\begin{equation}\label{eq:23}
  I_{(l+1)\uparrow,(m-1)\downarrow,n,p}\odot I_{l\uparrow,m\downarrow,n,p}
  \to
  I_{l\uparrow\uparrow,m\downarrow\downarrow,n,p}
\end{equation}

and its mirror image

\begin{equation}\label{eq:24}
  I_{l,m,(n-1)\downarrow,(p+1)\uparrow}\odot I_{l,m,n\downarrow,p\uparrow}
  \to
  I_{l,m,n\downarrow\downarrow,p\uparrow\uparrow},
\end{equation}

\begin{equation}\label{eq:25}
  (I_{l\uparrow,(m-1)\downarrow,n-1,p}\odot I_{l,m\downarrow,n\downarrow,p})
  \oplus
  (I_{l+1,(m-1)\downarrow,n\downarrow,p}\odot I_{l\uparrow,m\downarrow,n,p})
  \to
  I_{l\uparrow,m\downarrow\downarrow,n\downarrow,p}
\end{equation}

and its mirror image 

\begin{equation}\label{eq:26}
  (I_{l,m-1,(n-1)\downarrow,p\uparrow}\odot I_{l,m\downarrow,n\downarrow,p})
  \oplus
  (I_{l,m\downarrow,(n-1)\downarrow,p+1}\odot I_{l,m,n\downarrow,p\uparrow})
  \to
  I_{l,m\downarrow,n\downarrow\downarrow,p\uparrow},
\end{equation}

and finally, the self-dual

\begin{equation}\label{eq:27}
  (I_{l+1,m-1,n\downarrow,p\uparrow}\odot I_{l\uparrow,m\downarrow,n,p}
)  \oplus
  (I_{l\uparrow,m\downarrow,n-1,p+1}\odot I_{l,m,n\downarrow,p\uparrow})
  \to
  I_{l\uparrow,m\downarrow,n\downarrow,p\uparrow}.
\end{equation}

The nine `$\odot$' symbols above account for the nine possible ways of composing two morphisms, each of one of the three flavors listed in \Cref{le:deg1}. 

\begin{remark}
  Note that in all cases the product `$\odot$' conserves the total number of up as well as down arrows.
\end{remark}


\pf{th:sl-univ}
\begin{th:sl-univ}[sketch]
  As in the proof of \cite[Theorem 3.23]{us0}, an appeal to \cite[Theorem 2.22]{us0} together with \Cref{pr.itis-ordered} proves the statement as soon as we argue that the initial data of
  \begin{align*}
    x\hookrightarrow x_*^*, ~x_*\hookrightarrow x^*\text{ and}~p:x^*\otimes x_*^*\to {\bf 1}
  \end{align*}
  in the tensor category $\cD$ extends to a linear monoidal functor
  \begin{equation*}
    F:{}_I\bT'\to \cD,
  \end{equation*}
  where ${}_I\bT'$ is, as in \Cref{le:fullonx}, the full subcategory of ${}_I\bT$ on the objects $I_{l,m,n,p}$. Set ${}_I T:=T\otimes I$ for any $T\in \bT$. Since the objects of ${}_I\bT'$ are precisely the tensor powers (over $I\in {}_I\bT$) and the morphisms are tensor products and compositions of permutation of tensorands, evaluations \Cref{eq:pairi}, inclusions ${}_IV_*\subset {}_IV^*$ and ${}_IV\subset {}_IV_*^*$, etc., there is an obvious candidate for such an extension $F$, sending
  \begin{align*}
    {}_IV_*^*&\mapsto x_*^*\quad\text{and}\quad {}_IW\mapsto x_*^*/x,\\
    {}_IV^*&\mapsto x^*\quad\text{and}\quad {}_IW_*\mapsto x^*/x_*,\\
    \text{\Cref{eq:pairi}}&\mapsto p,
  \end{align*}
  etc. What we have to argue is that that extension is in fact well defined.

  The fact that, by \Cref{pr:isquad}, the algebra $\mathcal{A}$ defined in Section \ref{eq:a} is quadratic, means that it will be enough to check that the degree-two relations between degree-1 morphisms between the $I_{l,m,n,p}$ (i.e. the kernels of the maps \Cref{eq:22,eq:23,eq:24,eq:25,eq:26,eq:27}) vanish in $\cD$ upon substituting $x$ for ${}_IV$, $x_*$ for ${}_IV_*$, etc. This would be a somewhat tedious and unenlightening check if done exhaustively, so we exemplify the argument by treating \Cref{eq:22} alone. In that regard, we make the claim:

 {\it The kernel of the composition \Cref{eq:22} is generated, as an $(S_{l,m-2,n-2,p},S_{l,m,n,p})$-bimodule, by
    \begin{equation}\label{eq:inker}
      \phi_{m-1,n-1}\otimes \phi_{m,n} - \phi_{m-1,n-1}\otimes \phi_{m,n}\circ (m,m-1)(n,n-1),
    \end{equation}
    where $(m,m-1)$ is the respective transposition in $S_m\subset S_{l,m,n,p}$ and similarly,
    \begin{equation*}
      (n,n-1)\in S_n\subset S_{l,m,n,p}. 
    \end{equation*}
  }
  
  Assuming the claim for now, we observe that the relations annihilated by \Cref{eq:22} hold in any tensor category. It follows that our candidate functor $F$ is indeed compatible with the quadratic relations imposed by composition and hence is well defined. It thus remains to prove the claim; this is the goal we focus on for the duration of the present proof, following the layout of the proof of \cite[Lemma 3.27 (a)]{us0}.

  First, note that the morphism \Cref{eq:22} is surjective by \Cref{le:fullonx}. Secondly, the fact that \Cref{eq:inker} belongs to the kernel of \Cref{eq:22} is immediate: this is because 
  \begin{itemize}
  \item[$-$] evaluating the $m^{th}$ tensorand ${}_IV^*$ against the $n^{th}$ tensorand ${}_IV_*^*$, and then
  \item[$-$] evaluating the $(m-1)^{st}$  tensorand ${}_IV^*$ against the $(n-1)^{st}$  tensorand ${}_IV_*^*$
  \end{itemize}

  has the same effect as

  \begin{itemize}
  \item[$-$] permuting the $m^{th}$ and $(m-1)^{st}$ tensorands ${}_IV^*$,
  \item[$-$] permuting the $n^{th}$ and $(n-1)^{st}$ tensorands ${}_IV_*^*$
  \end{itemize}
  and then repeating the two evaluations above.

  The proof will thus be complete if we argue that the kernel of \Cref{eq:22} is not {\it strictly} larger than the bimodule generated by \Cref{eq:inker}. We do this by a dimension count. Tensoring two instances of \Cref{le:deg1}, \Cref{item:16} over $\bC [S_{l,m-1,n-1,p}]$, we conclude that the domain
  \begin{equation*}
    I_{l,(m-1)\downarrow,(n-1)\downarrow,p}\odot I_{l,m\downarrow,n\downarrow,p}
  \end{equation*}
  of \Cref{eq:22} is isomorphic to $\bC [S_{l,m,n,p}]$ as an $(\bC[S_{l,m-2,n-2,p}],\bC[S_{l,m,n,p}])$-bimodule, with
  \begin{itemize}
  \item[$-$] $\phi_{m,n}$ identified with the generator $1\in \bC [S_{l,m,n,p}]$, and
  \item[$-$] $S_{l,m-2,n-2,p}\subset S_{l,m,n,p}$ being the subgroup fixing $m$, $m-1$, $n$ and $n-1$. 
  \end{itemize}

  This identification turns the putative generator \Cref{eq:inker} of the kernel of \Cref{eq:22} into
  \begin{equation}\label{eq:elem}
    1-(m,m-1)(n,n-1)\in \bC [S_{l,m,n,p}].
  \end{equation}
  The $(\bC[S_{l,m-2,n-2,p}],\bC[S_{l,m,n,p}])$-bimodule generated by \Cref{eq:elem} coincides with the right $S_{l,m,n,p}$-module generated by the same element. The dimension of that module is half that of $\bC [S_{l,m,n,p}]$, and hence
  \begin{equation*}
    \dim \ker ~\text{\Cref{eq:22}} = \frac 12 l!m!n!p! = \frac 12 \left(\dim~\text{domain of \Cref{eq:22}}\right).
  \end{equation*}
  The desired conclusion that the kernel of the surjection \Cref{eq:22} cannot contain the bimodule generated by \Cref{eq:inker} strictly will thus follow if we prove that
  \begin{equation*}
    \dim I_{l,m\downarrow\downarrow,n\downarrow\downarrow,p} = \dim \Hom(I_{l,m,n,p},I_{l,m-2,n-2,p})\ge \frac 12 l!m!n!p!{.}
  \end{equation*}

  Since we have an embedding
  \begin{equation*}
    (\End I_{l,0,0,p})\otimes I_{0,m\downarrow\downarrow,n\downarrow\downarrow,0}\to I_{l,m\downarrow\downarrow,n\downarrow\downarrow,p}
  \end{equation*}
  and the left-hand tensorand is isomorphic to $\bC [S_{l,p}]$ by \Cref{le:deg0}, it is enough to assume that $l=p=0$ and show that
  \begin{equation*}
    \dim \Hom_{{}_I\bT}\left(({}_IV^*)^{\otimes m}\otimes_I ({}_IV_*^*)^{\otimes n}, ({}_IV^*)^{\otimes (m-2)}\otimes_I ({}_IV_*^*)^{\otimes (n-2)}\right)
    \ge
    \frac 12 m! n!
  \end{equation*}
  or equivalently, via the adjunction \Cref{eq:tit}, that 
    \begin{equation*}
    \dim \Hom_{\bT}\left((V^*)^{\otimes m}\otimes (V_*^*)^{\otimes n}, I\otimes (V^*)^{\otimes (m-2)}\otimes (V_*^*)^{\otimes (n-2)}\right)
    \ge
    \frac 12 m! n!.
  \end{equation*}  
  This, however, follows by restricting the morphisms on the left to $V_*^{\otimes m}\otimes V^{\otimes n}$ and noting that we already know the analogous inequality
  \begin{equation*}
    \dim \Hom_{\bT}\left(V_*^{\otimes m}\otimes V^{\otimes n}, V_*^{\otimes (m-2)}\otimes V^{\otimes (n-2)}\right)
    \ge
    \frac 12 m! n!
  \end{equation*}
  from the computation carried out in \cite[Lemma 6.3]{DPS}, or from \cite[Lemma 3.27 (a)]{us0} (which is analogous to the claim being proven here).
\end{th:sl-univ}

\section{Orthogonal and symplectic analogues of the categories $\bT$ and ${}_I\bT$}\label{se:bcd}

In this final section we discuss briefly the orthogonal and symplectic versions of the categories $\bT$ and $_I \bT$. The orthogonal and symplectic analogues of the Lie algebra $\fgl^M(V,V_*)$ are the Lie algebras $\fo(V)$ and $\fs\fp(V)$ where $V$ is now equipped with a nondegenerate symmetric or antisymmetric bilinear form $\langle \cdot, \cdot \rangle:V\times V \rightarrow \mathbb{C}$, yielding a respective linear map $S^2 V\rightarrow \bC$ or $\Lambda^2 V \rightarrow \bC$. The Lie algebras $\fo(V)$ and $\fs\fp(V)$ are defined as the respective largest subalgebras of $\fgl^M(V,V)$ for which the map $S^2 V\rightarrow \bC$ or $\Lambda^2 V \rightarrow \bC$ is a morphism of representations.

Let $\fg=\fo(V)$, $\fs\fp(V)$. Then $V$ is a submodule of $V^*$ (via the form $\langle \cdot,\cdot\rangle$), and the $\fg$-module $W:=V^*/V$ is irreducible. This can be proved for instance by considering $W$ over the family of Lie subalgebras $\fgl^M(V',V'_*)\subset \fg$ for varying decompositions of $V$ as $V'\oplus V'_{*}$ for maximal isotropic subspaces $V'$, $V'_{*}$. Over each such subalgebra $W$ is isomorphic to \begin{equation*}
V'^*/V_*\oplus (V_*)^*/V
\end{equation*} and hence has precisely two proper submodules. Since this submodules vary when $V'$ and $V'_*$ vary, the module $W$ is irreducible over $\fg$.

 Furthermore, for any Young diagram $\lambda$, the irreducible $\fgl^M(V,V)$-module $V_{\lambda}$ restricts to $\fg$ yielding a generally reducible $\fg$-module. In all cases the socle of $V_{\lambda}|_{\fg}$ is simple, and we denote it by $V_{[\lambda]}$ for $\fg=\fo(V)$ and by $V_{\langle\lambda \rangle}$ for $\fg=\fs\fp(V)$. It is clear that the Lie algebras $\fo(\infty)$ and $\fs\fp(\infty)$ considered in \cite{PS} are subalgebras respectively of $\fo(V)$ and $\fs\fp(V)$, and by \cite[Theorem 7.10]{PS2} the socle filtrations of $V_{\lambda}|_{\fo(\infty)}$ and $V_{\lambda}|_{\fs\fp(\infty)}$, described explicitly in \cite{PS}, coincide with the respective socle filtrations of $V_{\lambda}|_{\fo(V)}$ and $V_{\lambda}|_{\fs\fp(V)}$.

If $\lambda$, $\mu$ is a pair of Young diagrams, we set \begin{equation*} L_{\lambda,\mu}:=\begin{cases} W_{\lambda}\otimes V_{[\mu]}\text{ for } \fg=\fo(V), \\ W_{\lambda}\otimes V_{\langle\mu \rangle}\text{ for } \fg=\fs\fp(V).
\end{cases}\end{equation*}
Then $L_{\lambda,\mu}$ is a simple $\fg$-module. This can be seen by essentially the same argument as in the case of $W$. 
Moreover, \begin{equation*}
L_{\lambda,\mu} \cong L_{\lambda',\mu'} \text{ if and only if } \lambda=\lambda' \text{ and } \mu=\mu'.
\end{equation*} 

The analogue of the injective object $I$ from Section \ref{subse:ord-cat} is constructed as follows. One sets
\begin{equation*}
  F_{\fg}:=\begin{cases}
    S^2 W\text{ for }\fg=\fo(V), \\
    \Lambda^2 W\text{ for }\fg=\fs\fp(V).
  \end{cases}
\end{equation*}
Furthermore, the quotient $Q_{\fg}$ of $S^2V^*$ by the sum of kernels of the pairings $V^*\otimes V\rightarrow \bC$ and $S^2 V\rightarrow \bC$ admits a non-splitting exact sequence
\begin{equation*}
0 \rightarrow \bC \rightarrow Q_{\fg} \rightarrow F_{\fg} \rightarrow 0.
\end{equation*}

The socle filtration of $I$ has the form \begin{equation*}
\begin{tikzpicture}[auto,baseline=(current  bounding  box.center)]
\path (0,0) node[name=1, minimum width=1cm, rectangle split, rectangle split parts=3, draw, rectangle split draw splits=true, align=center] 
{ $S^2F_{\fg}$
	\nodepart{two} $F_{\fg}$
	\nodepart{three} $\bC$
} + (0,1.5) node {$\vdots$};
\path (0.7,-0.95) node {$.$};

\end{tikzpicture}  
\end{equation*}
Then the embedding (\ref{eq:skqsubsetsk+1q}) induces embeddings \begin{equation*}
S^kQ_{\fg} \hookrightarrow S^{k+1} Q_{\fg},
\end{equation*} which allow us to define $I_{\fg}$ as the colimit 
\begin{equation*}
I_{\fg}=\varinjlim S^kQ_{\fg}.
\end{equation*} 
Moreover, by the same construction as in Section \ref{subse:intalg} $I_{\fg}$ is endowed with the structure of a commutative algebra. 

The category $\bT_{\fg}$ is introduced in the same way as in Section \ref{subse:ord-cat}, where now $J_s=W^{\otimes l} \otimes V^{*\otimes m}$ for pairs $s=(l,m)$, $l$, $m\in\bN$, and the object $I$ is replaced by $I_{\fg}$. In the Introduction we denoted this category by $\bT^2_{\fg}$ to emphasize that is generated as a tensor category by two modules $V$ and $V^*$. In the rest of the paper we use the shorter notation $\bT_{\fg}$. We leave it to the reader to check that Proposition \ref{pr.is-ordered} holds also for the category $\bT_{\fg},$  and that $I_{\fg}$ is an injective hull in $\bT_{\fg}$ of the object $\bC$. The respective partial order $(l,m)\preceq (l',m')$ on $\bN\times\bN$ is given by $l\geq l'$, $m\leq m'$, $l+m'\leq l'+m$. The results of Section \ref{subse:endo} also hold with obvious modification.

The canonical injective resolution (\ref{eq:15}) stays the same with $F$ replaced by $F_{\fg}$, however now the socle of the object $\left(I_{\fg}\right)_j=I_{\fg}\otimes \Lambda^j F_{\fg} $ decomposes as $\bigoplus  \bS_{\lambda} V$ for $\fg=\fs\fp(V)$ and $\bigoplus  \bS_{\lambda^{\perp}} V$ for $\fg=\fo(V)$ where $\lambda$ runs over all special partitions of degree $j$.

\begin{corollary}
	For any simple object $X$ of $\bT_{\fo(V)}$ we have \begin{equation*}
	\operatorname{Ext}_{\bT_{\fo(V)}}^j (X,\bC)=\begin{cases}	0 \textup{ if }X\not\cong L_{\lambda,\varnothing} \textup{ for  a special } \lambda\textup{  with } |\lambda|=j, \\ \bC \textup{ if }X\cong L_{\lambda,\varnothing} \textup{ for  a special } \lambda\textup{  with } |\lambda|=j,
	\end{cases}
	\end{equation*}
	and for any simple object $X$ of $\bT_{\fs\fp(V)}$ we have \begin{equation*}
	\operatorname{Ext}_{\bT_{\fs\fp(V)}}^j (X,\bC)=\begin{cases}	0 \textup{ if }X\not\cong L_{\lambda^{\perp},\varnothing} \textup{ for  a special } \lambda\textup{  with } |\lambda|=j, \\ \bC \textup{ if }X\cong L_{\lambda^{\perp},\varnothing} \textup{ for  a special } \lambda\textup{  with } |\lambda|=j.
	\end{cases}
	\end{equation*}
\end{corollary}

Next, Theorem \ref{thm:right-def} and Proposition \ref{pr.thk} stay valid with $\bT$ replaced by $\bT_{\fg}$. We leave it to the reader to modify Lemma \ref{le.exact} accordingly. Furthermore, Proposition \ref{pr:ss}, and Theorem \ref{th:itsimp} also hold for $I_{\fg}$ and $Q_{\fg}$ (instead of $I$ and $Q$, respectively). The same applies to Proposition \ref{pr.itis-ordered}, Theorem \ref{thm:ihulls} (with $L_{\lambda,\mu}$ instead of $L_{\lambda,\mu,\nu,\pi}$), Theorem \ref{th:ikoszcoalg}, Theorem \ref{th:ikosz}, and Corollary \ref{cor.dgradedkoszul}.

The universality results from Section \ref{subse:univ} also carry over to the cases $\fg=\fo(V)$, $\fs\fp(V)$. In particular, the category ${}_{I_{\fg}}\bT$ is defined in the same way as the category ${}_{I}\bT$: it is the category of internal $I_{\fg}$-modules in $\bT_{\fg}$. 

Note also that the analogue \begin{equation*}
V^*\otimes V^* \rightarrow Q_{\fg}\subset I_{\fg}
\end{equation*}
of the map (\ref{eq:16}) is well defined and factors through maps \begin{equation*}
S^2V^*\rightarrow Q_{\fg}\text{ and } \Lambda^2V^*\rightarrow Q_{\fg}
\end{equation*}
in the respective cases $\fg=\fo(V)$ and $\fg=\fs\fp(V)$. This defines pairings \begin{equation}\label{eq:igsv}
I_{\fg}\otimes V^* \otimes_{I_{\fg}} \otimes I_{\fg} \otimes V^* \rightarrow I_{\fg} \otimes S^2V^* \rightarrow I_{\fg} 
\end{equation}
and
\begin{equation}\label{eq:iglambdav}
I_{\fg}\otimes V^* \otimes_{I_{\fg}} \otimes I_{\fg} \otimes V^* \rightarrow I_{\fg} \otimes \Lambda^2V^* \rightarrow I_{\fg},
\end{equation}
respectively. 

Now we have 

\begin{theorem}\label{th:osp-univ}
	Let $(\cD,\ \otimes,\ {\bf 1})$ be a tensor category, and $x\hookrightarrow x^*$ be a monomorphism in $\cD$. Assume that a morphism in $\cD$ 
	\begin{equation*}
	p: x^* \otimes x^*\to {\bf 1}
	\end{equation*}
 is given, satisfying $p\circ\sigma=p$ for $\fg=\fo(V)$ and $p\circ\sigma=-p$ for $\fg=\fs\fp(V)$, where $\sigma$ is the flip automorphism of $x^*\otimes x^*$ as an object of the tensor category $\cD$ coming from the assumption that $\cD$ is a symmetric monoidal category.
	\begin{enumerate}[(a)]
		\item There is a unique (up to monoidal natural isomorphism) left exact symmetric monoidal functor $F:{}_{I_{\fg}}\bT_{fin}\to \cD$, where $\fg=\fo(V)$ if $p\circ\sigma=p$ and $\fg=\fs\fp(V)$ if $p\circ\sigma=-p$, sending
		\begin{itemize}
			\item[$-$] the respective pairing (\ref{eq:igsv}) or (\ref{eq:iglambdav}) to $p$;
			\item[$-$] the surjection $V^*\to W$ to $x^*\to x^*/x$.
		\end{itemize}
		\item if $\cD$ is additionally a Grothendieck category then $F$ extends uniquely to a coproduct-preserving functor ${}_{I_{\fg}}\bT\to \cD$.
	\end{enumerate}
\end{theorem}

The universality of the tensor categories ${}_{I_{\fo(V)}}\bT$ and ${}_{I_{\fs\fp(V)}}\bT$ leads to the fact that they are equivalent as monoidal categories. More precisely, consider the (symmetric) tensor category $\bT^-_{\fs\fp(V)}$ defined in the same way as $\bT_{\fs\fp(V)}$ but with the flip isomorphism \begin{equation*}
\sigma: V^*\otimes V^* \rightarrow V^*\otimes V^*,~\sigma(v\otimes w)=w\otimes v
\end{equation*} 
replaced by $-\sigma$. One checks that $\bT^-_{\fs\fp(V)}$ is well-defined, i.e. that the new flip isomorphism on $V^*\otimes V^*$ induces a well-defined structure of tensor category preserving the monoidal structure on $\bT_{\fs\fp(V)}$.

In addition, one checks that there is a well-defined tensor category  ${}_{I_{\fs\fp(V)}}\bT^-$ of internal $I$-modules in $\bT^-_{\fs\fp(V)}$ which coincides with ${}_{I_{\fs\fp(V)}}\bT$ as a monoidal category.

\begin{corollary}
	The tensor categories ${}_{I_{\fo(V)}}\bT$ and ${}_{I_{\fs\fp(V)}}\bT^-$ are canonically equivalent.
\end{corollary}

\begin{proof}
	By Theorem \ref{th:osp-univ}, there are distinguished functors \begin{equation*}
	F:{}_{I_{\fo(V)}}\bT\rightarrow {}_{I_{\fs\fp(V)}}\bT^-
	\end{equation*}
	\begin{equation*}
	F^-:{}_{I_{\fs\fp(V)}}\bT^-\rightarrow {}_{I_{\fo(V)}}\bT
	\end{equation*}
	sending $V^* \otimes I_{\fo(V)}$ to $V^*\otimes I_{\fs\fp(V)}$, $V_* \otimes I_{\fo(V)}$ to $V_* \otimes I_{\fs\fp(V)}$ and $W \otimes I_{\fo(V)}$ to $W \otimes I_{\fs\fp(V)}$, and vice versa. Again, by Theorem \ref{th:osp-univ} $F$ and $F^-$ must be mutually inverse.
	
\end{proof}


\def\cftil#1{\ifmmode\setbox7\hbox{$\accent"5E#1$}\else
  \setbox7\hbox{\accent"5E#1}\penalty 10000\relax\fi\raise 1\ht7
  \hbox{\lower1.15ex\hbox to 1\wd7{\hss\accent"7E\hss}}\penalty 10000
  \hskip-1\wd7\penalty 10000\box7}

\addcontentsline{toc}{section}{References}

\Addresses

\end{document}